\documentclass[reqno,11pt,centertags]{amsart}
\usepackage{amsmath,amsthm,amscd,amssymb,latexsym,esint,upref,stmaryrd,enumerate,verbatim,yfonts,bm,mathtools,comment,mathrsfs}
\usepackage{color}
\usepackage{mathrsfs}
\usepackage[margin=1in]{geometry}

\usepackage{hyperref}

\hypersetup{
    colorlinks=true,
    linkcolor=red,
    filecolor=magenta,      
    urlcolor=blue,
    }

\newcommand*{\mailto}[1]{\href{mailto:#1}{\nolinkurl{#1}}}

\newcommand{\bbN}{{\mathbb{N}}}

\newcommand{\bbR}{{\mathbb{R}}}
\newcommand{\bbC}{{\mathbb{C}}}
\newcommand{\bbZ}{{\mathbb{Z}}}

\newcommand{\cB}{{\mathcal B}}
\newcommand{\cH}{{\mathcal H}}

\newcommand{\tr}{\mbox{tr}}

\newcommand{\hil}{\mathcal{H}}

\newcommand{\free}{{(0)}}

\DeclareMathOperator{\sech}{sech}
\DeclareMathOperator{\csch}{csch}
\DeclareMathOperator{\sgn}{sgn}
\DeclareMathOperator{\dist}{dist}

 \newcommand{\Rpz}{\big(H_{\ell,\varphi,R} - zI_{L^2((-\ell,\ell))}\big)^{-1}}
 \newcommand{\Rdz}{\big(H_{\ell,D} - zI_{L^2((-\ell,\ell))}\big)^{-1}}

 \newcommand{\Rpzo}{\big(H_{\ell,\varphi,R}^{(0)} - zI_{L^2((-\ell,\ell))}\big)^{-1}}
 \newcommand{\Rdzo}{\big(H_{\ell,D}^{(0)} - zI_{L^2((-\ell,\ell))}\big)^{-1}}
 \newcommand{\Rpko}{\big(H_{\ell,\varphi,R}^{(0)} +k^2I_{L^2((-\ell,\ell))}\big)^{-1}}
 \newcommand{\Rdko}{\big(H_{\ell,D}^{(0)} +k^2I_{L^2((-\ell,\ell))}\big)^{-1}}
 \newcommand{\Rrzo}{\big(H^{(0)}-zI_{L^2(\mathbb{R})} \big)^{-1}}
 \newcommand{\Rrko}{\big(H^{(0)}+k^2I_{L^2(\mathbb{R})} \big)^{-1}}
 \newcommand{\Rpio}{\big(H_{\ell,\varphi,R}^{(0)} - iI_{L^2((-\ell,\ell))}\big)^{-1}}
 \newcommand{\Rdio}{\big(H_{\ell,D}^{(0)} - iI_{L^2((-\ell,\ell))}\big)^{-1}}
  \newcommand{\Rpzbo}{\big(H_{\ell,\varphi,R}^{(0)} - \overline{z}I_{L^2((-\ell,\ell))}\big)^{-1}}
   \newcommand{\Rrzbo}{\big(H^{(0)}-\overline{z}I_{L^2(\mathbb{R})} \big)^{-1}}
\DeclareMathOperator{\dom}{dom}
\DeclareMathOperator*{\stlim}{s-lim}
\newcommand{\dott}{\,\cdot\,}

\renewcommand{\backslash}{\setminus}

\newcommand{\loc}{\operatorname{loc}}

\newcommand{\AC}{\text{\textnormal{AC}}}
\newcommand{\SL}{\text{\textnormal{SL}}}
 
\newcommand{\no}{\notag}
\newcommand{\lb}{\label}

\newcommand{\ol}{\overline}

\newcommand{\bi}{\bibitem}
\newcommand{\opl}{\oplus_{\ell}}

\makeatletter
\def\theequation{\@arabic\c@equation}
 
\allowdisplaybreaks
\numberwithin{equation}{section}
 
\newtheorem{theorem}{Theorem}[section]
\newtheorem{proposition}[theorem]{Proposition}
\newtheorem{lemma}[theorem]{Lemma}
\newtheorem{corollary}[theorem]{Corollary}
\newtheorem{hypothesis}[theorem]{Hypothesis}
\newtheorem{example}[theorem]{Example}

\theoremstyle{definition}

\theoremstyle{remark}
\newtheorem{remark}[theorem]{Remark}

\begin{document}
 
\title[Weak Convergence of Spectral Shift Functions Revisted]{Weak Convergence of Spectral Shift Functions Revisited}
 
%%%%%%%%%%%%%%%%%%%%%%%%%%%%%%%%%%%%%%%%%%
%%%%%%%%%%%%%%%%%%%%%%%%%%%%%%%%%%%%%%%%%%
% Author Information
%%%%%%%%%%%%%%%%%%%%%%%%%%%%%%%%%%%%%%%%%%
%%%%%%%%%%%%%%%%%%%%%%%%%%%%%%%%%%%%%%%%%%

\author[C.\ Connard]{Carson Connard}
\address{Department of Mathematics,
Kansas State University, 
1228 N. Martin Luther King Jr Dr., Manhattan, KS 66506, USA}
\email{\mailto{carson30@ksu.edu}}
\author[B.\ Ingimarson]{Benjamin Ingimarson}
\address{Department of Mathematical Sciences,
Carnegie Mellon University, 5000 Forbes Avenue,
Pittsburgh, PA 15289, USA}
\email{\mailto{bwi@andrew.cmu.edu}}
\author[R.\ Nichols]{Roger Nichols}
\address{Department of Mathematics (Dept.~6956), The University of Tennessee at Chattanooga, 615 McCallie Ave., Chattanooga, TN 37403, USA}
\email{\mailto{Roger-Nichols@utc.edu}}
%\email{Roger-Nichols@utc.edu}
\urladdr{\url{https://sites.google.com/mocs.utc.edu/rogernicholshomepage/home}}
\author[A.\ Paul]{Andrew Paul}
\address{Department of Mathematics;
University of California, San Diego;
9500 Gilman Drive, La Jolla, CA 92093-0112, USA}
\email{\mailto{anp004@ucsd.edu}}
\urladdr{\url{https://anpaul.weebly.com}}
 
%%%%%%%%%%%%%%%%%%%%%%%%%%%%%%%%%%%%%%%%%%
%%%%%%%%%%%%%%%%%%%%%%%%%%%%%%%%%%%%%%%%%%
%%%%%%%%%%%%%%%%%%%%%%%%%%%%%%%%%%%%%%%%%%
%%%%%%%%%%%%%%%%%%%%%%%%%%%%%%%%%%%%%%%%%%
 
\date{\today}
 
%\dedicatory{Dedicated to...}
%\thanks{To appear in {\it Journal Name Here}.}
\@namedef{subjclassname@2020}{\textup{2020} Mathematics Subject Classification}

%%%%%%%%%%%%%%%%%%%%%%%%%%%%%%%%%%%%%%%%%%
%%%%%%%%%%%%%%%%%%%%%%%%%%%%%%%%%%%%%%%%%%
\begin{abstract}
We study convergence of the spectral shift function for the finite interval restrictions of a pair of full-line Schr\"odinger operators to an interval of the form $(-\ell,\ell)$ with coupled boundary conditions at the endpoints as $\ell\to \infty$ in the case when the finite interval restrictions are relatively prime to those with Dirichlet boundary conditions.  Using a Krein-type resolvent identity we show that the spectral shift function for the finite interval restrictions converges weakly to that for the pair of full-line Schr\"odinger operators as the length of the interval tends to infinity.
\end{abstract}
%%%%%%%%%%%%%%%%%%%%%%%%%%%%%%%%%%%%%%%%%%
%%%%%%%%%%%%%%%%%%%%%%%%%%%%%%%%%%%%%%%%%%

\maketitle

%%%%%%%%%%%%%%%%%%%%%%%%%%%%%%%
%%%%%%%%%%%%%%%%%%%%%%%%%%%%%%%
\section{Introduction} \lb{s1}
%%%%%%%%%%%%%%%%%%%%%%%%%%%%%%%
%%%%%%%%%%%%%%%%%%%%%%%%%%%%%%%

It is a classical fact that the one-dimensional self-adjoint Schr\"odinger operators $H_0=-d^2/dx^2$ and $H=-d^2/dx^2 + V$, with $V$ real-valued and Lebesgue integrable on $\bbR$, are resolvent comparable in the sense that their resolvent difference
\begin{equation}
\big(H-zI_{L^2(\bbR)}\big)^{-1} - \big(H_0-zI_{L^2(\bbR)}\big)^{-1}
\end{equation}
belongs to the trace class for all $z\in \bbC\setminus \bbR$.  Consequently, there is a corresponding unique real-valued spectral shift function $\xi(\dott;H,H_0)$ defined a.e.~on $\bbR$ that vanishes in a neighborhood of $-\infty$ and satisfies
\begin{equation}
\int_{-\infty}^{\infty}\frac{|\xi(\lambda;H,H_0)|}{1+\lambda^2}\,\rm{d}\lambda < \infty.
\end{equation}
Moreover, for a wide class of functions $f$ (see \eqref{A.24}--\eqref{A.25}), Krein's trace formula holds:
\begin{equation}
\tr_{L^2(\bbR)}{(f(H)-f(H_0))} = \int_{-\infty}^{\infty}f'(\lambda)\xi(\lambda;H,H_0)\,\rm{d}\lambda,
\end{equation}
where $\tr_{L^2(\bbR)}{(\dott)}$ denotes the trace functional.

When $H_0$ and $H$ are restricted to a finite interval, say $(-\ell,\ell)$ with $\ell\in \bbN$, to obtain the finite interval Schr\"odinger operators $H_{0,\ell}$ and $H_{\ell}$ (imposing a fixed common choice of self-adjoint boundary conditions at the endpoints $\pm \ell$, thereby ensuring the self-adjointness of $H_{0,\ell}$ and $H_{\ell}$), the finite interval restrictions $H_{0,\ell}$ and $H_{\ell}$ are also resolvent comparable, so a unique real-valued spectral shift function $\xi(\dott;H_{\ell},H_{\ell,0})$ with properties analogous to those of $\xi(\dott;H,H_0)$ exists.  A problem of interest, dating back at least to the work of Borovyk and Makarov \cite{BM12} (see also \cite{Bo08} and the much earlier, related work of Kirsch \cite{Ki87}) has been to study the modes of convergence of $\xi(\dott;H_{\ell},H_{0,\ell})$ to $\xi(\dott;H,H_0)$ as $\ell\to \infty$.  Since $\xi(\dott;H_{\ell},H_{0,\ell})$ is integer-valued, owing to the discrete nature of the spectra of $H_{\ell}$ and $H_{0,\ell}$, and $\xi(\dott;H,H_0)$ is continuous and nonconstant on $(0,\infty)$, pointwise convergence of $\xi(\dott;H_{\ell},H_{0,\ell})$ to $\xi(\dott;H,H_0)$ as $\ell\to \infty$ is immediately dismissed as impossible.  Borovyk and Makarov showed that a certain averaging is needed in order to obtain convergence.  More specifically, in \cite{BM12},  Borovyk and Makarov proved that when Dirichlet boundary conditions are imposed at the endpoints of the finite intervals, the sequence $\xi(\dott;H_{\ell},H_{0,\ell})$ converges to $\xi(\dott;H,H_0)$ weakly in the sense that
\begin{equation}\lb{1.4}
\lim_{\ell\to \infty}\int_{-\infty}^{\infty} g(\lambda)\xi(\lambda;H_{\ell},H_{0,\ell})\,\textrm{d}\lambda = \int_{-\infty}^{\infty}g(\lambda)\xi(\lambda;H,H_0)\,\textrm{d}\lambda,\quad g\in C_c(\bbR),
\end{equation}
where $C_c(\bbR)$ denotes the set of compactly supported continuous complex-valued functions on $\bbR$.  To be as precise as possible, the finite intervals considered in \cite{BM12} are of the form $(0,\ell)$ and the infinite interval is the half-line $(0,\infty)$, but the arguments given in \cite{BM12} may be modified to treat the symmetric intervals $(-\ell,\ell)$ and the line $\bbR$.  Using an abstract approach based on modified Fredholm determinants, \eqref{1.4} was extended to all separated self-adjoint boundary conditions in \cite{GN12} (the abstract approach is separately developed in \cite{GN12b}).  Actually, in \cite{GN12}, \eqref{1.4} was improved to
\begin{equation}\lb{1.5}
\lim_{\ell\to \infty}\int_{-\infty}^{\infty}h(\lambda)\frac{\xi(\lambda;H_{\ell},H_{0,\ell})}{1+\lambda^2}\,\textrm{d}\lambda = \int_{-\infty}^{\infty}h(\lambda)\frac{\xi(\lambda;H,H_0)}{1+\lambda^2}\,\textrm{d}\lambda,\quad h\in C_b(\bbR),
\end{equation}
where $C_b(\bbR)$ denotes the set of bounded continuous complex-valued functions on $\bbR$.  However, \cite{GN12} does not discuss coupled self-adjoint boundary conditions at $\pm \ell$, which are of the form
\begin{equation}\lb{1.6}
\begin{pmatrix}
y(\ell)\\
y'(\ell)
\end{pmatrix}=
e^{i\varphi}R
\begin{pmatrix}
y(-\ell)\\
y'(-\ell)
\end{pmatrix}
\end{equation}
for a fixed parameter $\varphi\in[0,\pi)$ and a fixed matrix $R = (R_{j,k})_{1\leq j,k\leq 2}\in \bbR^{2\times 2}$ with $\det{(R)}=1$.  In \cite{MN17}, \eqref{1.5} was shown to hold for a special subclass of the coupled boundary conditions \eqref{1.6}, namely those for which $R_{1,2}=0$, so that $R$ has the special form
\begin{equation}\lb{1.7}
R=
\begin{pmatrix}
a & 0\\
b & a^{-1}
\end{pmatrix}
\end{equation}
for some $a,b\in \bbR$ with $a\neq 0$.  The approach taken in \cite{MN17} is based on an analysis of the coefficient in the Krein identity connecting the resolvent of the finite interval restriction with the boundary conditions \eqref{1.6}--\eqref{1.7} to the resolvent of the finite interval restriction with Dirichlet boundary conditions.  In the case of \eqref{1.7}, the difference of the resolvents of the coupled and Dirichlet restrictions is a rank one operator, so the Krein identity contains exactly one coefficient.  The work in \cite{MN17} leaves open the question of convergence---in particular, the validity of \eqref{1.5}---when $R_{1,2}\neq 0$.

For $R_{1,2}\neq 0$, the situation is slightly more delicate as the resolvent difference is then a rank two operator and the corresponding Krein identity is more easily understood in terms of four coefficients, see Theorem \ref{t2.6} below and \cite[Theorem 3.2 $(i)$]{CGNZ14}.  Our main aim here is to extend \eqref{1.5} to the case $R_{1,2}\neq 0$.  To do this, we carefully analyze the $\ell\to \infty$ behavior of the four coefficients in the Krein identity and use the abstract criteria for convergence of spectral shift functions developed in \cite{GN12b} to extend \eqref{1.5} to the case $R_{1,2}\neq 0$.  When combined with \cite{GN12} and \cite{MN17}, the results we obtain here show that \eqref{1.5} holds for all self-adjoint boundary conditions.

In Section \ref{s2}, we rigorously define the full-line Schr\"odinger operators $H$ and $H_0$ and their finite interval counterparts and recall some of their basic properties.  We also recall Krein's resolvent identity and separately investigate the behavior of its coefficients with respect to both the finite interval length and the spectral parameter.  The key results needed to invoke the abstract criteria from \cite{GN12} for convergence of spectral shift functions---in particular, convergence of various Birman--Schwinger-type operators in the trace and Hilbert--Schmidt classes---are established in Section \ref{s3}.  Finally, in Section \ref{s4}, we combine the results from Section \ref{s3} with the abstract convergence criteria from \cite{GN12b} to obtain our main result on convergence of spectral shift functions corresponding to coupled boundary conditions with $R_{1,2}\neq 0$.  For completeness, Appendix \ref{sA} recalls the convergence criteria from \cite{GN12b} in a form that is suitably tailored to the application required in Section 4.

\medskip

\noindent
\textbf{Notation:}  If $A \in \bbC^{m \times n}$ for some $m,n\in \bbN$, then the $(j,k)$-entry of $A$ is denoted by $A_{j,k}$. $\SL_2{(\bbR)}$ denotes the special linear group of order $2$ (i.e., the set of all matrices in $\bbR^{2\times2}$ with determinant equal to one).  If $[a,b]\subseteq \bbR$, then $\AC([a,b])$ denotes the set of all functions $f\colon [a,b]\to \bbR$ that are absolutely continuous.  If $(a,b)\subseteq \bbR$, then $\AC_{\loc}{((a,b))}$ denotes the set of all $f\colon (a,b)\to \bbC$ that are locally absolutely continuous on $(a,b)$ (i.e., absolutely continuous on every compact subinterval of $(a,b)$).  The symbol ``a.e.'' abbreviates the phrase ``almost everywhere (with respect to Lebesgue measure on $\bbR$),''  and the symbol ``$\coloneqq$'' means ``defined to be equal to.'' If $f\colon (a,b)\to\bbR$, then $f_{\pm} = (|f| \pm f)/2$ denote the positive and negative parts of $f$, respectively.  If $T\colon \dom{(T)}\subseteq \cH\to \cH$ is a closable linear operator in the Hilbert space $\cH$, then $\ol{T}$ denotes the closure of $T$.  If $T$ is a closed linear operator in a Hilbert space, then $\rho(T)$ and $\sigma(T)$ denote the resolvent set and spectrum of $T$, respectively.  If $J\subseteq \bbR$ is an open interval, then $H^k(J)=W^{k,2}(J)$ denotes the Sobolev space of order $k\in \bbN$.  If $\cH$ is a separable complex Hilbert space, then $\cB(\cH)$ and $\cB_p(\cH)$, $p\in [1,\infty)$, denote the set of bounded linear operators and the $\ell^p$-based Schatten--von Neumann trace ideals on $\cH$, respectively.  $C_c(\bbR)$, $C_b(\bbR)$, and $C_{\infty}(\bbR)$ denote the sets of continuous complex-valued functions on $\bbR$ that are compactly supported, bounded, and converge to zero at $\pm \infty$, respectively.

%%%%%%%%%%%%%%%%%%%%%%%%%%%%%%%
%%%%%%%%%%%%%%%%%%%%%%%%%%%%%%%
\section{One-Dimensional Schr\"odinger Operators} \lb{s2}
%%%%%%%%%%%%%%%%%%%%%%%%%%%%%%%
%%%%%%%%%%%%%%%%%%%%%%%%%%%%%%%

We begin by introducing the following hypothesis which shall be assumed throughout.

%%%%%%%%%%%
\begin{hypothesis}\lb{h2.1}
$V \in L^1(\bbR)$ is real-valued a.e., $\varphi \in [0,\pi)$, and $R\in \SL_2{(\bbR)}$ with $R_{1,2} \neq 0$.
\end{hypothesis}
%%%%%%%%%%%

In turn, the full-line Schr\"odinger operator with potential $V$ is defined by
\begin{equation}\lb{2.1}
Hf \coloneqq -f'' + Vf,\quad f\in\dom{(H)} \coloneqq \big\{ g\in L^2(\bbR)\,\big|\,g,g'\in\AC_{\loc}{(\bbR)},\, -g'' + Vg \in L^2(\bbR)\big\}.
\end{equation}
Since the condition $V \in L^1(\bbR)$ implies that the differential expression $\tau \coloneqq -d^2/dx^2 + V(x)$ is in the limit point case at $\pm \infty$, the operator $H$ is self-adjoint in the Hilbert space $L^2(\bbR)$ equipped with the inner product
\begin{equation}\lb{2.2}
\langle f,g\rangle_{L^2(\bbR)} \coloneqq \int_{-\infty}^{\infty}{\overline{f(x)}g(x)\, \mathrm{d}x},\quad f,g \in L^2(\bbR).
\end{equation}
Moreover, the sesquilinear form uniquely associated to $H$ is
\begin{equation}\lb{2.3}
\mathfrak{Q}[f,g] \coloneqq \int_{-\infty}^\infty \Big[ \overline{f'(x)}g'(x) + \overline{f(x)}V(x)g(x)\Big]\, \mathrm{d}x,\quad f,g \in \dom{(\mathfrak{Q})} \coloneqq H^1(\mathbb{R}).
\end{equation}

%%%%%%%%%%%
\begin{remark}\lb{r2.2}
In the case when $V=0$ a.e.~on $\bbR$, the corresponding ``free'' Schr\"odinger operator and its sesquilinear form are denoted by $H^{\free}$ and $\mathfrak{Q}^{\free}$, respectively.\hfill $\diamond$
\end{remark}
%%%%%%%%%%%

The following fixed quantities will play an important role:
\begin{equation}
M_V\coloneqq \int_{-\infty}^{\infty}V_-(x)\,\mathrm{d}x,\quad N_R\coloneqq \frac{|R_{1,1}| + |R_{2,2}| + 2}{|R_{1,2}|}.
\end{equation}
In fact, standard estimates imply that $H$ is lower semibounded with a lower bound expressed in terms of the quantity $M_V$.

%%%%%%%%%%%
\begin{lemma}\lb{l2.3}
If Hypothesis \ref{h2.1} holds, then
\begin{equation}\lb{2.4}
\mathfrak{Q}[f,f] \geq -M_V(1+M_V)\|f\|_{L^2(\bbR)}^2,\quad f \in H^1(\bbR).
\end{equation}
In particular, if $V_-=0$ a.e.~on $\bbR$, then $H$ is nonnegative.
\end{lemma}
%%%%%%%%%%
\begin{proof}
Let $f\in H^1(\bbR)$.  Writing $V = V_+ - V_-$, one obtains
\begin{align}\lb{2.5}
\mathfrak{Q}[f,f] &= \int_{-\infty}^\infty \Big[ |f'(x)|^2 + V(x)|f(x)|^2\Big]\, \mathrm{d}x \geq \|f'\|_{L^2(\bbR)}^2 - \int_{-\infty}^{\infty}V_-(x)|f(x)|^2\, \mathrm{d}x.
\end{align}
If $M_V=0$, then $V_-=0$ a.e.~on $\bbR$ and \eqref{2.5} implies $\mathfrak{Q}[f,f] \geq 0$, which is \eqref{2.4} for $M_V=0$.  Thus, $H$ is nonnegative in this case.

If $M_V>0$, then one bounds the second integral in \eqref{2.5} from above in terms of $\|f'\|_{L^2(\bbR)}^2$ as follows.  By \cite[Lemma 9.32]{Te14}, for every $\varepsilon > 0$ and every $n \in \bbZ$,
\begin{equation}\lb{2.6}
\sup_{x\in[n,n+1]} |f(x)|^2 \leq \varepsilon\int_n^{n+1} |f'(x)|^2\, \mathrm{d}x + \left(1+\frac{1}{\varepsilon}\right)\int_n^{n+1} |f(x)|^2\, \mathrm{d}x \eqqcolon B_f(\varepsilon,n),
\end{equation}
and the scalars $B_f(\varepsilon,n)$ satisfy
\begin{align}\lb{2.7}
    \sum_{n\in\mathbb{Z}} B_f(\varepsilon,n)&=\varepsilon\|f'\|_{L^2(\bbR)}^2+\left(1+\frac{1}{\varepsilon} \right)\|f\|_{L^2(\bbR)}^2,\quad \varepsilon > 0.
\end{align}
Therefore, by \eqref{2.6} and \eqref{2.7},
\begin{align}
    \int_{-\infty}^{\infty} V_-(x)|f(x)|^2\,\mathrm{d}x &=\sum_{n\in\mathbb{Z}} \int_n^{n+1} V_-(x)|f(x)|^2\,  \mathrm{d}x\no\\
    &\leq\sum_{n\in\mathbb{Z}} B_f(\varepsilon,n)\int_n^{n+1} V_-(x)\, \mathrm{d}x\no\\
    &\leq \varepsilon M_V\|f'\|_{L^2(\bbR)}^2+M_V\bigg(1+\frac{1}{\varepsilon} \bigg)\|f\|_{L^2(\bbR)}^2,\quad \varepsilon > 0.\lb{2.8}
\end{align}
Choosing $\varepsilon=1/M_V$ in \eqref{2.8} yields
\begin{equation}\lb{2.9}
\int_{-\infty}^{\infty} V_-(x)|f(x)|^2\,\mathrm{d}x \leq \|f'\|_{L^2(\bbR)}^2+M_V\left(1+M_V\right)\|f\|_{L^2(\bbR)}^2.
\end{equation}
Finally, applying \eqref{2.9} in \eqref{2.5} yields \eqref{2.4}.
\end{proof}
%%%%%%%%%%%

Factoring the potential coefficient $V$ according to
\begin{equation}
V = uv,\quad u \coloneqq \sgn{(V)}|V|^{1/2},\quad v \coloneqq |V|^{1/2}\, \text{ a.e.~on $\bbR$},
\end{equation}
the following well-known trace ideal properties of the resolvent of $H^{\free}$ when multiplied by the factors $u$ and $v$ hold (see, e.g., \cite[(4.19) and (4.20)]{GN12b} and \cite[(2.49) and (2.50)]{MN17}).
%%%%%%%%%%%
\begin{lemma}\lb{l2.4a}
If Hypothesis \ref{h2.1} holds, then
\begin{equation}\lb{2.12d}
\begin{split}
u\Rrzo,\overline{\Rrzo v}\in \cB_2\big(L^2(\bbR)\big),&\\
\overline{u\Rrzo v}\in \cB_1\big(L^2(\bbR)\big),&\quad z\in \bbC\backslash[0,\infty).
\end{split}
\end{equation}
In addition,
\begin{equation}
\lim_{z\to -\infty}\Big\|\ol{u\Rrzo v}\Big\|_{\cB_1(L^2(\bbR))}=0.
\end{equation}
\end{lemma}
%%%%%%%%%%%

%%%%%%%%%%%
\begin{remark}\lb{r2.5}
Lemma \ref{l2.4a} with $V$ replaced by $|V|$ yields
\begin{equation}
\begin{split}
&|V|^{1/2}\big(H^{\free}+I_{L^2(\bbR)}\big)^{-1/2}\ol{\big(H^{\free}+I_{L^2(\bbR)}\big)^{-1/2}|V|^{1/2}}\\
&\quad = \ol{|V|^{1/2}\big(H^{\free}+I_{L^2(\bbR)}\big)^{-1}|V|^{1/2}}\in \cB_1\big(L^2(\bbR)\big),
\end{split}
\end{equation}
which implies $|V|^{1/2}\big(H^{\free}+I_{L^2(\bbR)}\big)^{-1/2}\in \cB_2\big(L^2(\bbR)\big)$.  Therefore, $|V|$ is relatively form compact, hence infinitesimally form bounded, with respect to $H^{\free}$.  In particular, the positive and negative parts of $V$ are infinitesimally form bounded with respect to $H^{\free}$. \hfill$\diamond$
\end{remark}
%%%%%%%%%%%

For each $\ell \in \bbN$, let
\begin{equation}
V_{\ell} \coloneqq V|_{(-\ell,\ell)}    
\end{equation}
denote the restriction of $V$ to $(-\ell,\ell)$ and introduce the finite-interval Schr\"odinger operator $H_{\ell,\varphi,R}$ with coupled boundary conditions at the endpoints of $(-\ell,\ell)$ by
\begin{align}
&H_{\ell,\varphi,R}f \coloneqq -f'' + V_{\ell}f,\lb{2.10}\\
&f \in \dom{(H_{\ell,\varphi,R})} \coloneqq \bigg\{g \in L^2((-\ell,\ell))\,\bigg|\, g,g' \in \AC([-\ell,\ell]),\, 
\begin{pmatrix}
g(\ell)\\
g'(\ell)
\end{pmatrix}
=e^{i\varphi}R
\begin{pmatrix}
g(-\ell)\\
g'(-\ell)
\end{pmatrix},\no\\
&\hspace*{10cm} -g''+V_{\ell}g\in L^2((-\ell,\ell))
\bigg\}.\no
\end{align}
The operator $H_{\ell,\varphi,R}$ is self-adjoint (see, e.g., \cite[Theorem 2.5]{CGNZ14}) in the Hilbert space $L^2((-\ell,\ell))$ equipped with the inner product
\begin{equation}
\langle f,g \rangle_{L^2((-\ell,\ell))} \coloneqq \int_{-\ell}^{\ell} \overline{f(x)}g(x)\, \mathrm{d}x,\quad f,g \in L^2((-\ell,\ell)).
\end{equation}
The sesquilinear form uniquely associated to $H_{\ell,\varphi,R}$ is (see, e.g., \cite[Section 3.3 and (3.3.143)]{GNZ22})
\begin{align}
&\mathfrak{Q}_{\ell,\varphi,R}[f,g] \coloneqq \int_{-\ell}^{\ell}\Big[\overline{f'(x)}g'(x) + \overline{f(x)}V_{\ell}(x)g(x)\Big]\, \mathrm{d}x\lb{2.12}\\
&\hspace*{2.3cm}- \frac{1}{R_{1,2}} \Big[R_{1,1} \overline{f(-\ell)}g(-\ell) - e^{-i\varphi} \overline{f(-\ell)}g(\ell)- e^{i\varphi} \overline{f(\ell)}g(-\ell) + R_{2,2}\overline{f(\ell)}g(\ell) \Big],\no\\
&f,g \in \dom{(\mathfrak{Q}_{\ell,\varphi,R})} \coloneqq H^1((-\ell,\ell)).\no%= \big\{h\in L^2((-\ell,\ell))\,\big|\, h\in \AC([-\ell,\ell]),\, h'\in L^2((-\ell,\ell))\big\}
\end{align}
Since the differential expression $\tau_{\ell} \coloneqq -d^2/dx^2 + V_{\ell}(x)$ is regular on $(-\ell,\ell)$, the self-adjoint operator $H_{\ell,\varphi,R}$ is lower semibounded.  A careful analysis of the sesquilinear form $\mathfrak{Q}_{\ell,\varphi,R}$ shows that $H_{\ell,\varphi,R}$ is lower semibounded uniformly with respect to $\ell \in \bbN$.

%%%%%%%%%%%
\begin{lemma}\lb{l2.4}
If Hypothesis \ref{h2.1} holds, then for every $\ell\in \bbN$,
\begin{equation}\lb{2.13}
\mathfrak{Q}_{\ell,\varphi,R}[f,f] \geq -(1+M_V+N_R)\|f\|_{L^2((-\ell,\ell))}^2,\quad f \in H^1((-\ell,\ell)).
\end{equation}
\end{lemma}
%%%%%%%%%%%
\begin{proof}
        Let $\ell \in \bbN$ and $f\in H^1((-\ell,\ell))$. By \eqref{2.12},
\begin{align}\lb{2.14}
    \mathfrak{Q}_{\ell,\varphi,R}[f,f] &= \underbrace{\int_{-\ell}^\ell \Big[|f'(x)|^2 + V_{\ell}(x) |f(x)|^2\Big]\, \mathrm dx}_{\eqqcolon I_{\ell}[f]}\\
    &\quad - \underbrace{\frac{1}{R_{1,2}} \Big[R_{1,1} f(-\ell) \overline{f(-\ell)} - e^{-i\varphi} \overline{f(-\ell)}f(\ell) - e^{i\varphi} \overline{f(\ell)}f(-\ell) + R_{2,2}f(\ell)\overline{f(\ell)} \Big]}_{\eqqcolon J_{\ell,\varphi,R}[f]}.\no
\end{align}
The integral $I_{\ell}[f]$ is estimated from below as follows:
\begin{align}\lb{2.15}
    I_{\ell}[f] &= \int_{-\ell}^\ell |f'(x)|^2\, \mathrm dx + \int_{-\ell}^\ell [V_+(x) - V_-(x)] |f(x)|^2 \,\mathrm dx\no\\
    &\geq \|f'\|_{L^2((-\ell,\ell))}^2 - \int_{-\ell}^\ell V_-(x)|f(x)|^2\, \mathrm dx\no\\
    &\geq \|f'\|_{L^2((-\ell,\ell))}^2 - M_V\sup_{x \in [-\ell,\ell]}{|f(x)|^2}.
\end{align}
In addition, the boundary terms $J_{\ell,\varphi,R}[f]$ satisfy:
\begin{align}
    J_{\ell,\varphi,R}[f] &= \frac{1}{R_{1,2}} \Big[ R_{1,1}|f(-\ell)|^2 + R_{2,2}|f(\ell)|^2\Big] - \frac{1}{R_{1,2}}\Big[e^{-i\varphi}\overline{f(-\ell)}f(\ell) + e^{i\varphi} \overline{f(\ell)}f(-\ell) \Big]\no\\
    &\leq \frac{1}{|R_{1,2}|} \Big[ |R_{1,1}||f(-\ell)|^2 + |R_{2,2}||f(\ell)|^2 + 2|f(\ell)||f(-\ell)| \Big]\no\\
    &\leq \frac{1}{|R_{1,2}|}\Big[|R_{1,1}| |f(-\ell)|^2 + |R_{2,2}||f(\ell)|^2 + |f(-\ell)|^2 + |f(\ell)|^2 \Big]\no\\
    &\leq N_R\sup_{x \in [-\ell,\ell]}{|f(x)|^2}.\lb{2.16}
\end{align}
Applying \eqref{2.15} and \eqref{2.16} in \eqref{2.14}, one obtains:
\begin{equation}\lb{2.17}
    \mathfrak{Q}_{\ell, \varphi,R}[f,f] = I_{\ell}[f] - J_{\ell,\varphi,R}[f]\geq \|f'\|_{L^2((-\ell,\ell))}^2 - (M_V+N_R) \sup_{x \in [-\ell,\ell]}{|f(x)|^2}.
\end{equation}
Note that $\sup_{x \in [-\ell,\ell]}|f(x)|^2 = \sup_{x \in [n,n+1]}|f(x)|^2$ for some $n = n(f) \in [-\ell,\ell-1] \cap \mathbb{N}$. By \cite[Lemma 9.32]{Te14},
\begin{align}
    \sup_{x \in [-\ell,\ell]}|f(x)|^2 &= \sup_{x \in [n, n+1]}|f(x)|^2 \no\\
    &\leq \varepsilon \int_n^{n+1} |f'(x)|^2\, \mathrm dx + \left( 1 + \frac{1}{\varepsilon} \right) \int_n^{n+1} |f(x)|^2\, \mathrm dx \no\\
    &\leq \varepsilon \|f'\|_{L^2((-\ell,\ell))}^2 +\left( 1 + \frac{1}{\varepsilon}\right) \|f\|_{L^2((-\ell,\ell))}^2,\quad \varepsilon > 0.\lb{2.18}
\end{align}
Finally, taking $\varepsilon = (M_V+N_R)^{-1}>0$ in \eqref{2.18} and applying the resulting estimate in \eqref{2.17} yields \eqref{2.13}. 
 Neither $M_V$ nor $N_R$ depend on $\ell \in \bbN$, so \eqref{2.13} implies that $H_{\ell,\varphi,R}$ is lower semibounded uniformly with respect to $\ell \in \bbN$.
\end{proof}
%%%%%%%%%%%

For each $\ell \in \bbN$, the finite-interval Schr\"odinger operator with Dirichlet boundary conditions at the endpoints of $(-\ell,\ell)$ is defined by
\begin{align}
&H_{\ell,D}f \coloneqq -f'' + V_{\ell}f,\\
&f\in \dom{(H_{\ell,D})} \coloneqq \big\{g\in L^2((-\ell,\ell))\,\big|\,g,g'\in\AC([-\ell,\ell]),\, g(-\ell)=g(\ell)=0,\no\\
&\hspace*{8.1cm}-g''+V_{\ell}g\in L^2((-\ell,\ell))\big\}.\no
\end{align}
The operator $H_{\ell,D}$ is self-adjoint in $L^2((-\ell,\ell))$ for each $\ell\ \in \bbN$, and its corresponding sesquilinear form is
\begin{align}
&\mathfrak{Q}_{\ell,D}[f,g] \coloneqq \int_{-\ell}^{\ell} \Big[\ol{f'(x)}g'(x) + \ol{f(x)}V_{\ell}(x)g(x)\Big]\,\mathrm{d}x,\\
&f,g\in\dom{(\mathfrak{Q}_{\ell,D})} \coloneqq \big\{h\in L^2((-\ell,\ell))\,\big|\,h\in\AC([-\ell,\ell]),\,h(-\ell)=h(\ell)=0,\,h'\in L^2((-\ell,\ell))\big\}\no\\
&\hspace*{2.94cm}= H_0^1((-\ell,\ell)).\no
\end{align}

%%%%%%%%%%%
\begin{remark}
$(i)$ The condition $R_{1,2}\neq 0$ in Hypothesis \ref{h2.1} implies that $H_{\ell,D}$ and $H_{\ell,\varphi,R}$ are \textit{relatively prime} with respect to their underlying minimal Sturm--Liouville operator:
\begin{align}
&H_{\ell,\min}f\coloneqq-f''+V_{\ell}f,\\
&f\in\dom{(H_{\ell,\min})}\coloneqq\big\{g\in L^2((-\ell,\ell))\,\big| \, g,g'\in \AC([-\ell,\ell]),\, g(-\ell)=g'(-\ell)=g(\ell)=g'(\ell)=0,\no\\
&\hspace*{11.15cm}-g''+V_{\ell}g\in L^2((-\ell,\ell))\big\},\no
\end{align}
in the sense that $\dom{\big(H_{\ell,D}\big)}\cap\dom{\big(H_{\ell,\varphi,R}\big)}=\dom{\big(H_{\ell,\min}\big)}$.  $(ii)$ In the case when $V=0$ a.e.~on $\bbR$, the corresponding ``free'' finite-interval Schr\"odinger operators and their sesquilinear forms will be denoted by $H^{\free}_{\ell,\varphi,R}$, $H^{\free}_{\ell,D}$, $\mathfrak{Q}^{\free}_{\ell,\varphi,R}$, and $\mathfrak{Q}^{\free}_{\ell,D}$, respectively, for each $\ell \in \bbN$.  In particular, one infers that $H^{\free}_{\ell,D}$ is nonnegative for every $\ell\in \bbN$.\hfill $\diamond$
\end{remark}
%%%%%%%%%%%

The resolvent operators of $H_{\ell,\varphi,R}$ and $H_{\ell,D}$ are related for each $\ell \in \bbN$ via a Krein-type resolvent identity.  To wit, for each $z\in \rho(H_{\ell,D})$, let $\{\psi_{\ell,m}(z,\dott)\}_{m=1,2}$ denote solutions to the Schr\"odinger differential equation $-\psi''+V_{\ell}\psi=z\psi$ on $(-\ell,\ell)$ that satisfy the boundary conditions
\begin{equation}\lb{2.21}
\begin{split}
&\psi_{\ell,1}(z,-\ell) = 0,\quad \psi_{\ell,1}(z,\ell) = 1,\\
&\psi_{\ell,2}(z,-\ell) = 1,\quad \psi_{\ell,2}(z,\ell) = 0.
\end{split}
\end{equation}
Krein's resolvent identity relates the resolvent operators of $H_{\ell,\varphi,R}$ and $H_{\ell,D}$ via a rank-two operator constructed from the solutions $\{\psi_{\ell,m}(z,\dott)\}_{m=1,2}$.
%%%%%%%%%%%
\begin{theorem}[{\cite[Theorem 3.2]{CGNZ14}}]\lb{t2.6}
Assume Hypothesis \ref{h2.1}.  If $\ell\in \bbN$ and $z\in \rho(H_{\ell,\varphi,R})\cap \rho(H_{\ell,D})$, then the matrix
\begin{equation}\lb{2.22}
K_{\ell,\varphi,R}(z) \coloneqq
\begin{pmatrix}
\frac{R_{2,2}}{R_{1,2}} - \psi_{\ell,1}'(z,\ell)   &   \frac{-1}{e^{-i\varphi}R_{1,2}} - \psi_{\ell,2}'(z,\ell)   \\
\frac{-1}{e^{i\varphi}R_{1,2}} + \psi_{\ell,1}'(z,-\ell)   &   \frac{R_{1,1}}{R_{1,2}} + \psi_{\ell,2}'(z,-\ell)   
\end{pmatrix}
\end{equation}
is invertible and
\begin{equation}\lb{2.27a}
\Rpz = \Rdz + P_{\ell,\varphi,R}(z),
\end{equation}
where the rank-two operator $P_{\ell,\varphi,R}(z)$ is defined by
\begin{equation}\lb{2.28b}
P_{\ell,\varphi,R}(z) \coloneqq - \sum_{m,n=1}^2 \big[K_{\ell,\varphi,R}(z)^{-1}\big]_{m,n}\langle \psi_{\ell,n}(\ol{z},\dott),\dott\rangle_{L^2((-\ell,\ell))}\psi_{\ell,m}(z,\dott).
\end{equation}
\end{theorem}
%%%%%%%%%%%

%%%%%%%%%%%
\begin{remark}
In the case when $V=0$ a.e.~on $\bbR$, the solutions corresponding to \eqref{2.21}, the matrix \eqref{2.22}, and the rank-two operator \eqref{2.28b} will be denoted by $\psi_{\ell,m}^{\free}(z,\dott)$, $m\in \{1,2\}$, $K_{\ell,\varphi,R}^{\free}(z)$, and $P_{\ell,\varphi,R}^{\free}(z)$, respectively.\hfill $\diamond$
\end{remark}
%%%%%%%%%%%

Introducing for each $\ell\in \bbN$ the factorization of the potential coefficient $V_{\ell}$,
\begin{equation}
V_{\ell} = u_{\ell}v_{\ell},\quad u_{\ell} \coloneqq \sgn{(V_{\ell})}|V_{\ell}|^{1/2},\quad v_{\ell} \coloneqq |V_{\ell}|^{1/2}\, \text{ a.e.~on $(-\ell,\ell)$},
\end{equation}
the following analogue of Lemma \ref{l2.4a} for the free Dirichlet operator $H_{\ell,D}^{\free}$ holds (see, e.g., \cite[(2.69), (3.12), and (3.14)]{GN12}).

%%%%%%%%%%%
\begin{lemma}\lb{l2.9}
If Hypothesis \ref{h2.1} holds and $\ell\in \bbN$, then
\begin{equation}
\begin{split}
u_{\ell}\Rdzo,\overline{\Rdzo v_{\ell}}\in \cB_2\big(L^2((-\ell,\ell))\big),&\\
\overline{u_{\ell}\Rdzo v_{\ell}}\in \cB_1\big(L^2((-\ell,\ell))\big),&\quad z\in\rho\big(H_{\ell,D}^{\free}\big).
\end{split}
\end{equation}
\end{lemma}
%%%%%%%%%%%

The analogue of Lemma \ref{l2.9} for $H_{\ell,\varphi,R}^{\free}$ is obtained as a consequence of Theorem \ref{t2.6}.

%%%%%%%%%%%
\begin{lemma}\lb{l2.10}
If Hypothesis \ref{h2.1} holds and $\ell\in \bbN$, then
\begin{align}
u_{\ell}\Rpzo,\overline{\Rpzo v_{\ell}}\in \cB_2\big(L^2((-\ell,\ell))\big),&\lb{2.29b}\\
\overline{u_{\ell}\Rpzo v_{\ell}}\in \cB_1\big(L^2((-\ell,\ell))\big),&\quad z\in\rho\big(H_{\ell,\varphi,R}^{\free}\big).\lb{2.30b}
\end{align}
\end{lemma}
%%%%%%%%%%%
\begin{proof}
Let $\ell \in \bbN$.  Taking $z=i$ (and $V_{\ell}\equiv 0$), \eqref{2.27a} implies
\begin{equation}\lb{2.31a}
\begin{split}
u_{\ell}\Rpio &= u_{\ell}\Rdio\\
&\quad - \sum_{m,n=1}^2 \Big[K_{\ell,\varphi,R}^{\free}(i)^{-1}\Big]_{m,n}\big\langle \psi_{\ell,n}^{\free}(-i,\dott),\dott\big\rangle_{L^2((-\ell,\ell))}u_{\ell}\psi_{\ell,m}^{\free}(i,\dott).
\end{split}
\end{equation}
Since $u_{\ell}\psi_{\ell,m}^{\free}(i,\dott)\in L^2((-\ell,\ell))$, $m\in \{1,2\}$, the sum on the right-hand side in \eqref{2.31a} is a rank-two operator in $L^2((-\ell,\ell))$.  Thus, the right-hand side of \eqref{2.31a} belongs to $\cB_2\big(L^2((-\ell,\ell))\big)$ by Lemma \ref{l2.9}.  If $z\in \rho\big(H_{\ell,\varphi,R}^{\free}\big)$, the first resolvent identity implies
\begin{equation}
\begin{split}
&u_{\ell}\Rpzo\\
&\quad = u_{\ell}\Rpio\Big[I_{L^2((-\ell,\ell))} + (z-i)\Rpzo\Big],
\end{split}
\end{equation}
which is a product of operators in $\cB_2\big(L^2((-\ell,\ell))\big)$ and $\cB\big(L^2((-\ell,\ell))\big)$, respectively, and therefore belongs to $\cB_2\big(L^2((-\ell,\ell))\big)$.  The first containment in \eqref{2.29b} follows.  The same argument \textit{mutatis mutandis} shows
\begin{equation}\lb{2.33a}
v_{\ell}\Rpzo \in \cB_2\big(L^2((-\ell,\ell))\big),\quad z\in \rho\big(H_{\ell,\varphi,R}^{\free}\big).
\end{equation}
In consequence, for each $z\in \rho\big(H_{\ell,\varphi,R}^{\free}\big)$,
\begin{equation}
\ol{\Rpzo v_{\ell}} = \Big[v_{\ell}\big(H_{\ell,\varphi,R}^{\free} - \ol{z}I_{L^2((-\ell,\ell))}\big)^{-1} \Big]^* \in \cB_2\big(L^2((-\ell,\ell))\big),
\end{equation}
which is the second containment in \eqref{2.29b}.

Multiplying both sides of \eqref{2.31a} by $v_{\ell}$ from the right and taking closures yields
\begin{align}
\ol{u_{\ell}\Rpio v_{\ell}} &= \ol{u_{\ell}\Rdio v_{\ell}}\lb{2.36a}\\
&\quad - \sum_{m,n=1}^2 \Big[K_{\ell,\varphi,R}^{\free}(i)^{-1}\Big]_{m,n}\big\langle v_{\ell}\psi_{\ell,n}^{\free}(-i,\dott),\dott\big\rangle_{L^2((-\ell,\ell))}u_{\ell}\psi_{\ell,m}^{\free}(i,\dott).\no
\end{align}
Since $v_{\ell}\psi_{\ell,m}^{\free}(-i,\dott),u_{\ell}\psi_{\ell,m}^{\free}(i,\dott)\in L^2((-\ell,\ell))$, $j\in \{1,2\}$, the sum on the right-hand side in \eqref{2.36a} is a rank-two operator in $L^2((-\ell,\ell))$.  Thus, the right-hand side of \eqref{2.36a} belongs to $\cB_1\big(L^2((-\ell,\ell))\big)$ by Lemma \ref{l2.9}.  If $z\in \rho\big(H_{\ell,\varphi,R}^{\free}\big)$, the first resolvent identity implies
\begin{align}\lb{2.37a}
&\ol{u_{\ell}\Rpzo v_{\ell}}\\
&\quad = \ol{u_{\ell}\Rpio v_{\ell}}+ (z-i) u_{\ell}\Rpio \ol{\Rpzo v_{\ell}}.\no
\end{align}
The first term on the right-hand side in \eqref{2.37a} was shown above to belong to $\cB_1\big(L^2((-\ell,\ell))\big)$.  The second term on the right-hand side in \eqref{2.37a} belongs to $\cB_1\big(L^2((-\ell,\ell))\big)$ since, by \eqref{2.29b}, it is a product of two operators in $\cB_2\big(L^2((-\ell,\ell)\big)$.  The containment in \eqref{2.30b} follows.
\end{proof}
%%%%%%%%%%%

%%%%%%%%%%%
\begin{remark}\lb{r2.12}
If $\ell\in \bbN$, and $k>(1+N_R)^{1/2}$, then upon replacing $V_{\ell}$ by $|V_{\ell}|$, Lemma \ref{l2.10} implies
\begin{equation}
\begin{split}
&|V_{\ell}|^{1/2}\big(H_{\ell,\varphi,R}^{\free}+k^2I_{L^2((-\ell,\ell))}\big)^{-1/2}\ol{\big(H_{\ell,\varphi,R}^{\free}+k^2I_{L^2((-\ell,\ell))}\big)^{-1/2}|V_{\ell}|^{1/2}}\\
&\quad = \ol{|V_{\ell}|^{1/2}\big(H_{\ell,\varphi,R}^{\free}+k^2I_{L^2((-\ell,\ell))}\big)^{-1}|V_{\ell}|^{1/2}}\in \cB_1\big(L^2((-\ell,\ell))\big),
\end{split}
\end{equation}
which implies $|V_{\ell}|^{1/2}\big(H_{\ell,\varphi,R}^{\free}+k^2I_{L^2((-\ell,\ell))}\big)^{-1/2}\in \cB_2\big(L^2((-\ell,\ell))\big)$.  Therefore, $|V_{\ell}|$ is relatively form compact, hence infinitesimally form bounded, with respect to $H_{\ell,\varphi,R}^{\free}$.  In particular, the positive and negative parts of $V_{\ell}$ are infinitesimally form bounded with respect to $H_{\ell,\varphi,R}^{\free}$ for $\ell\in \bbN$. \hfill$\diamond$
\end{remark}
%%%%%%%%%%%

In the applications below, the results of Theorem \ref{t2.6} are most often needed in the free case $V=0$ a.e.~on $\bbR$ and for $z=-k^2$ with $k>0$ taken sufficiently large.  In the free case and for negative values of $z$, the solutions determined by \eqref{2.21} have the form
\begin{equation}
\psi_{\ell,m}^{\free}(-k^2,x) = \frac{1}{2}\bigg[\frac{\cosh{(kx)}}{\cosh{(k\ell)}} + (-1)^{m-1}\frac{\sinh{(kx)}}{\sinh{(k\ell)}} \bigg],\quad x\in [-\ell,\ell],\, m\in \{1,2\},\, k\in(0,\infty),
\end{equation}
and the matrix \eqref{2.22} reduces to
\begin{align}
K_{\ell,\varphi,R}^{\free}(-k^2)=
\begin{pmatrix}
\frac{R_{2,2}}{R_{1,2}} - \frac{k}{2}[\coth{(k\ell)} + \tanh{(k\ell)}]   &   \frac{-1}{e^{-i\varphi}R_{1,2}} - \frac{k}{2}[\tanh{(k\ell)}-\coth{(k\ell)}]\\
\frac{-1}{e^{i\varphi}R_{1,2}} - \frac{k}{2}[\tanh{(k\ell)} - \coth{(k\ell)}]   &   \frac{R_{1,1}}{R_{1,2}} - \frac{k}{2}[\coth{(k\ell)} + \tanh{(k\ell)}]
\end{pmatrix},&\lb{2.26}\\
k\in\big((1+N_R)^{1/2},\infty\big).&\no
\end{align}
Lemma \ref{l2.4} ensures that $z=-k^2 \in \rho\big(H_{\ell,\varphi,R}^{\free}\big)$ for all $\ell \in \bbN$ when $k > (1 + N_R)^{1/2}$.  Therefore, since $H_{\ell,D}^{\free}$ is nonnegative for all $\ell\in \bbN$, the condition on $k$ in \eqref{2.26} ensures that $z=-k^2 \in \rho\big(H_{\ell,\varphi,R}^{\free}\big)\cap \rho\big(H_{\ell,D}^{\free}\big)$ for all $\ell \in \bbN$.

The limiting behaviors of the entries of the inverse matrix $K_{\ell,\varphi,R}^{\free}(-k^2)^{-1}$ as $k\to \infty$ with $\ell\in \bbN$ held fixed and as $\ell\to \infty$ with $k$ taken sufficiently large and held fixed will also play an important role.  Upon inspection, \eqref{2.26}, implies
\begin{equation}\lb{2.27}
\det\Big(K_{\ell,\varphi,R}^{\free}(-k^2)\Big) \underset{k\to \infty}{\sim} k^2\, \text{ for each fixed $\ell\in \bbN$}.
\end{equation}
Taking the inverse of the matrix in \eqref{2.26} and applying the relation in \eqref{2.27} yields an asymptotic estimate for the entries of $K_{\ell,\varphi,R}^{\free}(-k^2)^{-1}$:
\begin{equation}\lb{2.42}
\Big[K_{\ell,\varphi,R}^{\free}(-k^2)^{-1}\Big]_{m,n} \underset{k\to \infty}{=}O(1/k)\,\text{ for each fixed $\ell\in \bbN$}.
\end{equation}
In addition, for each fixed $k\in \big((1+N_R)^{1/2},\infty\big)$, \eqref{2.26} and $\det{(R)}=1$ imply
\begin{align}
\lim_{\ell\to \infty}\det\Big(K_{\ell,\varphi,R}^{\free}(-k^2)\Big) = k^2- \frac{R_{1,1}+R_{2,2}}{R_{1,2}}k + \frac{R_{2,1}}{R_{1,2}} \eqqcolon p_{R}^{\free}(k).
\end{align}
Elementary algebraic manipulations, again using $\det{(R)}=1$, reveal that the discriminant of the polynomial $p_R^{\free}$ is positive.  If $k_R^{\free}$ denotes the largest root of $p_R^{\free}$, then
\begin{equation}\lb{2.30}
\lim_{\ell \to \infty}\det{\left(K_{\ell,\varphi,R}^{\free}(-k^2)\right)} \neq 0\, \text{ if $k>\max{\left\{k_R^{\free},(1+N_R)^{1/2}\right\}}$}.
\end{equation}
For fixed $k$, each entry of the matrix in \eqref{2.26} has a finite limit as $\ell \to \infty$.  Consequently, \eqref{2.30} implies that each entry of $K_{\ell,\varphi,R}^{\free}(-k^2)^{-1}$ has a finite limit as $\ell \to \infty$.  In particular,
\begin{equation}\lb{2.31}
\Big[K_{\ell,\varphi,R}^{\free}(-k^2)^{-1}\Big]_{m,n} \underset{\ell\to \infty}{=} O(1)\, \text{ for $m,n\in\{1,2\}$ and each fixed $k>\max{\left\{k_R^{\free},(1+N_R)^{1/2}\right\}}$}.
\end{equation}

%%%%%%%%%%%
\begin{example}
In general, the discriminant of $p_R^{\free}$ is $\big[(R_{1,1}-R_{2,2})^2 + 4\big]/R_{1,2}^2$, so the largest root of the polynomial $p_R^{\free}$ is
\begin{equation}
k_R^{\free} = \frac{1}{2}\bigg\{\frac{R_{1,1}+R_{2,2}}{R_{1,2}} + \frac{[(R_{1,1}-R_{2,2})^2+4]^{1/2}}{|R_{1,2}|} \bigg\}.
\end{equation}
Choosing
\begin{equation}
R=R_{\alpha}\coloneqq
\begin{pmatrix}
0 & \alpha\\
-\alpha^{-1} & 0
\end{pmatrix},
\end{equation}
where $\alpha \in \bbR\backslash\{0\}$ is a parameter, results in
\begin{equation}
k_{R_{\alpha}}^{\free} = \frac{1}{|\alpha|},\quad (1+N_{R_{\alpha}})^{1/2} = \bigg(1+\frac{2}{|\alpha|} \bigg)^{1/2}.
\end{equation}
The function
\begin{equation}
\frac{1}{|\alpha|} - \bigg(1+\frac{2}{|\alpha|} \bigg)^{1/2},\quad \alpha\in \bbR\backslash\{0\},
\end{equation}
exhibits a sign change at $\alpha = \sqrt{2}-1$.  Therefore, depending on $R$, $\max\Big\{k_R^{\free},(1+N_R)^{1/2}\Big\}$ can be either of $k_R^{\free}$ or $(1+N_R)^{1/2}$.
\end{example}
%%%%%%%%%%%

%%%%%%%%%%%%%%%%%%%%%%%%%%%%%%%
%%%%%%%%%%%%%%%%%%%%%%%%%%%%%%%
\section{Convergence Results for Resolvent Operators} \lb{s3}
%%%%%%%%%%%%%%%%%%%%%%%%%%%%%%%
%%%%%%%%%%%%%%%%%%%%%%%%%%%%%%%

The Hilbert spaces $L^2((-\ell,\ell))$ vary with $\ell \in \bbN$.  Therefore, we introduce the following $\ell$-dependent direct sum to move to the fixed $\ell$-independent Hilbert space $L^2(\bbR)$.  If $\ell \in \bbN$, $f\in L^2((-\ell,\ell))$, and $g\in L^2(\bbR\setminus(-\ell,\ell))$, then we define $f\opl g$ by
\begin{equation}
(f\opl g)(x)\coloneqq
\begin{cases}
f(x)& \text{for a.e.~$x\in(-\ell,\ell)$},\\
g(x)& \text{for a.e.~$x\in \bbR\setminus(-\ell,\ell)$}.
\end{cases}
\end{equation}
By additivity of the Lebesgue integral, $f\opl g \in L^2(\bbR)$ and
\begin{equation}
\|f\opl g\|_{L^2(\bbR)}^2 = \|f\|_{L^2((-\ell,\ell))}^2 + \|g\|_{L^2(\bbR\setminus (-\ell,\ell))}^2.
\end{equation}
Conversely, if $h\in L^2(\bbR)$, then $h$ can be expressed as
\begin{equation}
h = f \opl g,\, \text{ where $f=h|_{(-\ell,\ell)}$ and $g=h|_{\bbR\setminus (-\ell,\ell)}$}.
\end{equation}
In this way, for each $\ell\in \bbN$,
\begin{equation}\lb{3.4}
L^2(\bbR) = L^2((-\ell,\ell)) \opl L^2(\bbR\backslash(-\ell,\ell)).
\end{equation}
By additivity of the Lebesgue integral if $h_1,h_2\in L^2(\bbR)$ with
\begin{equation}
h_j = f_j \opl g_j,\quad f_j\in L^2((-\ell,\ell)),\quad g_j\in L^2(\bbR\backslash(-\ell,\ell)),\quad j\in \{1,2\},
\end{equation}
then
\begin{equation}
\langle h_1,h_2 \rangle_{L^2(\bbR)} = \langle f_1,f_2 \rangle_{L^2((-\ell,\ell))} + \langle g_1,g_2 \rangle_{L^2(\bbR\backslash(-\ell,\ell))}.
\end{equation}
Finally, if
\begin{equation}
S\colon \dom{(S)}\subseteq L^2((-\ell,\ell)) \to L^2((-\ell,\ell))
\end{equation}
and 
\begin{equation}
T\colon \dom{(T)}\subseteq L^2(\bbR\backslash(-\ell,\ell))\to L^2(\bbR\backslash(-\ell,\ell))
\end{equation}
are linear operators, then their direct sum $S\opl T$ with respect to the decomposition \eqref{3.4} is defined in the standard way by
\begin{equation}
(S\opl T)h\coloneqq (Sf)\opl (Tg),\quad h=f\opl g\in \dom{(S\opl T)} \coloneqq \dom{(S)} \opl \dom{(T)}.
\end{equation}

Under this direct sum formalism, it is known that the sequence $\left\{H_{\ell,D}^{(0)} \opl 0\right\}_{\ell = 1}^\infty$ converges to the free Schr\"odinger operator $H^{\free}$ in the strong resolvent sense as $\ell\to \infty$.

%%%%%%%%%%%
\begin{lemma}[Lemma 3.1 in \cite{GN12}]\lb{l3.1}
If Hypothesis \ref{h2.1} holds, then for each fixed $ z \in \mathbb{C} \setminus [0,\infty)$,
\begin{equation}\lb{3.10}
\stlim_{\ell \to \infty}{\left( \left[ H_{\ell,D}^{(0)} \opl 0 \right] - z I_{L^2(\mathbb{R})} \right)^{-1}} = \Rrzo.
\end{equation}
In particular, the sequence $\left\{H_{\ell,D}^{(0)} \opl 0\right\}_{\ell = 1}^\infty$ converges to $H^{(0)}$ in the strong resolvent sense.
\end{lemma}
%%%%%%%%%%%

One can apply Theorem \ref{t2.6} to establish an analogue of Lemma \ref{l3.1} for the operators $H_{\ell,\varphi,R}^{\free}$.  Our proof employs the following elementary abstract fact, the proof of which is an exercise, for computing the operator or trace ideal norm of a rank one operator.

%%%%%%%%%%%
\begin{proposition}\lb{pA.1}
Let $\cH$ denote a Hilbert space with inner product $\langle \dott, \dott\rangle_{\cH}$.  If $\psi,\phi\in \cH$ and the rank one operator $A\colon \cH\to \cH$ is given by $A=\langle\psi, \dott\rangle_{\cH} \phi$, then for each $p\in[1,\infty)$,
\begin{equation}
\|A\|_{\cB_p(\cH)} = \|\psi\|_{\cH}\|\phi\|_{\cH} = \|A\|_{\cB(\cH)}.
\end{equation}
\end{proposition}
%%%%%%%%%%%

The analogue of Lemma \ref{l3.1} for the operators $H_{\ell,\varphi,R}^{\free}$ is then stated as follows:

%%%%%%%%%%%
\begin{lemma}\lb{l3.2b}
If Hypothesis \ref{h2.1} holds and $k_{R,\max}\coloneqq \max\big\{k_R^{\free},(1+N_R)^{1/2}\big\}$, then for each fixed $ z \in \mathbb{C} \setminus [-k_{R,\max}^2,\infty)$,
\begin{equation}\lb{3.11b}
\stlim_{\ell \to \infty}{\left( \left[ H_{\ell,\varphi,R}^{(0)} \opl 0 \right] - z I_{L^2(\bbR)} \right)^{-1}} = \Rrzo.
\end{equation}
In particular, the sequence $\left\{H_{\ell,\varphi,R}^{(0)} \opl 0\right\}_{\ell = 1}^\infty$ converges to $H^{(0)}$ in the strong resolvent sense.
\end{lemma}
%%%%%%%%%%%
\begin{proof}
It suffices to show
\begin{equation}\lb{3.12b}
\stlim_{\ell \to \infty}{\left( \left[ H_{\ell,\varphi,R}^{(0)} \opl 0 \right] + k^2 I_{L^2(\bbR)} \right)^{-1}} = \Rrko
\end{equation}
for some $k > k_{R,\max}$, as the statement in \eqref{3.11b} then follows from an application of the first resolvent identity.  Let $k > k_{R,\max}$, so that $-k^2 \in \rho\big(H_{\ell,\varphi,R}^{\free}\big) \cap \rho\big(H_{\ell,D}^{\free}\big) \cap \rho\big(H^{\free}\big)$ for all $\ell \in \bbN$.  By Theorem \ref{t2.6},
\begin{equation}\lb{3.13ba}
\begin{split}
\Big(\Big[H_{\ell,\varphi,R}^{(0)}\opl 0\Big]+k^2I_{L^2(\mathbb{R})}\Big)^{-1} = \Big(\Big[H_{\ell,D}^{(0)}\opl 0\Big]+k^2I_{L^2(\bbR)}\Big)^{-1} + \Big[P_{\ell,\varphi,R}^{(0)}(-k^2)\opl 0\Big],&\\
\ell\in \bbN.&
\end{split}
\end{equation}
In light of \eqref{3.10} and \eqref{3.13ba}, in order to show \eqref{3.12b}, it suffices to prove
\begin{equation}\lb{3.14ba}
\stlim_{\ell\to \infty}{\Big[P_{\ell,\varphi,R}^{(0)}(-k^2)\opl 0\Big]} = 0.
\end{equation}
To this end, note that the sequence of operators on the left-hand side in \eqref{3.14ba} is uniformly bounded with respect to $\ell\in \bbN$.  Indeed, by the triangle inequality and \eqref{2.31},
\begin{align}
&\Big\|P_{\ell,\varphi,R}^{(0)}(-k^2)\oplus_{\ell}0\Big\|_{\mathcal{B}(L^2(\mathbb{R}))}\no\\
&\quad =\Big\|P_{\ell,\varphi,R}^{(0)}(-k^2)\Big\|_{\mathcal{B}(L^2((-\ell,\ell)))}\no\\
&\quad \leq \sum_{m,n=1}^2{\Big|\Big[K_{\ell,\varphi,R}(-k^2)^{-1}\Big]_{m,n}\Big|\ \big\|\big\langle \psi_{\ell,m}^{(0)}(-k^2,\dott),\dott\big\rangle \psi_{\ell,n}^{(0)}(-k^2,\dott)\big\|_{\mathcal{B}(L^2((-\ell,\ell)))}}\no\\
&\,\; \underset{\ell\to \infty}{\leq} O(1)\sum_{m,n=1}^2{\big\|\big\langle \psi_{\ell,m}^{(0)}(-k^2,\dott),\dott\big\rangle \psi_{\ell,n}^{(0)}(-k^2,\dott)\big\|_{\mathcal{B}(L^2((-\ell,\ell)))}}\no\\
&\,\; \underset{\ell\to \infty}{=}O(1)\sum_{m,n=1}^2{\big\|\psi_{\ell,m}^{(0)}(-k^2,\dott)\big\|_{L^2((-\ell,\ell))}\big\|\psi_{\ell,n}^{(0)}(-k^2,\dott)\big\|_{L^2((-\ell,\ell))}},\lb{3.15ba}
\end{align}
where the final equality follows from Proposition \ref{pA.1}.  One infers that
\begin{align}
\big\|\psi_{\ell,m}^{(0)}(-k^2,\dott)\big\|_{L^2((-\ell,\ell))}^2&=\frac{1}{4}\int_{-\ell}^{\ell}{\bigg[\frac{\cosh{(kx)}}{\cosh{(k\ell)}}\pm\frac{\sinh{(kx)}}{\sinh{(k\ell)}}\bigg]^2\, \mathrm{d}x}\no\\
&=\frac{1}{4}\int_{-\ell}^{\ell}\bigg[\frac{\cosh^2{(kx)}}{\cosh^2{(k\ell)}}+\frac{\sinh^2{(kx)}}{\sinh^2{(k\ell)}}\bigg]\, \mathrm{d}x,\quad m\in \{1,2\},\, \ell\in \bbN,\lb{3.16b}
\end{align}
where the third term that arises from the expansion of the square in the first integral vanishes as it is an odd function being integrated over a finite symmetric interval. The last integral in \eqref{3.16b} is elementary and one obtains:
\begin{equation}\lb{3.17b}
\begin{split}
\big\|\psi_{\ell,m}^{(0)}(-k^2,\dott)\big\|_{L^2((-\ell,\ell))}^2=\frac{\tanh{(k\ell)}+\coth{(k\ell)}}{4k}+\frac{\ell}{4}\big[\sech^2{(k\ell)}-\csch^2{(k\ell)}\big],&\\
m\in \{1,2\},\,\ell\in \bbN.&
\end{split}
\end{equation}
In particular, \eqref{3.17b} implies
\begin{equation}\lb{3.18b}
\big\|\psi_{\ell,m}^{(0)}(-k^2,\dott)\big\|_{L^2((-\ell,\ell))} \underset{\ell\to \infty}{=}O(1),\quad m\in \{1,2\}.
\end{equation}
Combining \eqref{3.15ba} and \eqref{3.18b}, one obtains
\begin{equation}
\Big\|P_{\ell,\varphi,R}^{(0)}(-k^2)\oplus_{\ell}0\Big\|_{\mathcal{B}(L^2(\mathbb{R}))} \underset{\ell\to \infty}{=} O(1),
\end{equation}
so that $\Big\{P_{\ell,\varphi,R}^{\free}(-k^2)\opl 0\Big\}_{\ell=1}^{\infty}$ is a bounded sequence in $\cB\big(L^2(\bbR)\big)$.  Therefore, by \cite[Exercise 4.28]{We80}, in order to show \eqref{3.14ba}, it suffices to prove
\begin{equation}\lb{3.20}
\lim_{\ell\to \infty}\Big[P_{\ell,\varphi,R}^{\free}(-k^2)\opl 0\Big]f = 0,\quad f\in L^1(\bbR)\cap L^2(\bbR),
\end{equation}
since $L^1(\bbR)\cap L^2(\bbR)$ is dense in $L^2(\bbR)$.  If $f\in L^1(\mathbb{R})\cap L^2(\mathbb{R})$, then one obtains
\begin{align}
\big|\chi_{(-\ell,\ell)}(x)\psi_{\ell,m}^{\free}(-k^2,x)f(x)\big|&=\frac{1}{2}\bigg|\chi_{(-\ell,\ell)}(x)\bigg[\frac{\cosh{(kx)}}{\cosh{(k\ell)}}+(-1)^{m-1}\frac{\sinh{(kx)}}{\sinh{(k\ell)}}\bigg]f(x)\bigg|\no\\
&\leq |f(x)|\,\text{ for a.e.~$x\in \bbR$},\quad \ell\in \bbN,\, m\in\{1,2\}.
\end{align}
Since $\lim_{\ell\to\infty}{\psi_{\ell,m}^{\free}(-k^2,x)}=0$ for every $x\in\mathbb{R}$, it follows that
\begin{equation}
\lim_{\ell\to\infty}{\chi_{(-\ell,\ell)}(x)\psi_{\ell,m}^{\free}(-k^2,x)f(x)}=0\,\text{ for a.e.~$x\in \bbR$}.
\end{equation}
Hence, by the dominated convergence theorem,
\begin{equation}\lb{3.23}
\begin{split}
\lim_{\ell\to\infty}\big\langle \psi_{\ell,m}^{\free}(-k^2,\dott),f\big|_{(-\ell,\ell)}\big\rangle_{L^2((-\ell,\ell))}=\lim_{\ell\to\infty}{\int_{-\infty}^{\infty}{\chi_{(-\ell,\ell)}\psi_{\ell,m}^{\free}(-k^2,x)f(x)\, \mathrm{d}x}}=0,&\\
m\in\{1,2\}.&
\end{split}
\end{equation}
Finally, applying \eqref{2.28b}, \eqref{2.31}, \eqref{3.18b}, and \eqref{3.23}, one obtains
\begin{align}
&\Big\|\Big[P_{\ell,\varphi,R}^{\free}(-k^2) \opl 0 \Big]f\, \Big\|_{L^2(\bbR)}\\
&\quad = \Big\|P_{\ell,\varphi,R}^{\free}(-k^2)f\big|_{(-\ell,\ell)}\, \Big\|_{L^2((-\ell,\ell))}\no\\
&\quad \leq \sum_{m,n=1}^2 \Big|\Big[K_{\ell,\varphi,R}^{\free}(-k^2)^{-1}\Big]_{m,n}\Big|\Big|\big\langle \psi_{\ell,n}^{\free}(-k^2,\dott),f\big|_{(-\ell,\ell)}\big\rangle_{L^2((-\ell,\ell))}\Big|\big\|\psi_{\ell,m}^{\free}(-k^2,\dott)\big\|_{L^2((-\ell,\ell))}\no\\
&\,\; \underset{\ell\to \infty}{=} o(1),\no
\end{align}
and \eqref{3.20} follows.
\end{proof}
%%%%%%%%%%%

The following result, which studies the $z\to -\infty$ behavior of the extended Birman--Schwinger operator of $H_{\ell,D}^{\free}$ with $\ell\in \bbN$ fixed, is an immediate consequence of \cite[Lemma 3.8]{GN12}.

%%%%%%%%%%%
\begin{lemma}\lb{l3.3}
If Hypothesis \ref{h2.1} holds, then
\begin{equation}\lb{3.25c}
\lim_{z \to -\infty} \left\| \ol{u_\ell \Rdzo v_\ell} \oplus_\ell 0 \right\|_{\cB_1(L^2(\bbR))} = 0,\quad \ell\in \bbN.
\end{equation}
\end{lemma}
%%%%%%%%%%%

Using the Krein identity \eqref{2.27a}, \eqref{2.28b}, and the large parameter asymptotics \eqref{2.42}, one establishes the following analogue of Lemma \ref{l3.3} for the operator $H_{\ell,\varphi,R}^{\free}$.

%%%%%%%%%%%
\begin{lemma}\lb{l3.4}
If Hypothesis \ref{h2.1} holds, then
\begin{equation}\lb{3.26c}
\lim_{z \to -\infty} \left\| \ol{u_\ell \Rpzo v_\ell} \oplus_\ell 0 \right\|_{\cB_1(L^2(\bbR))} = 0,\quad \ell\in \bbN.
\end{equation}
\end{lemma}
%%%%%%%%%%%
\begin{proof}
Let $\ell \in \bbN$.  It suffices to show
\begin{equation}\lb{3.27c}
\lim_{k \to \infty} \left\| \overline{u_\ell \Rpko v_\ell} \right\|_{\cB_1(L^2((-\ell,\ell)))} = 0.
\end{equation}
By Theorem \ref{t2.6}, for $k>k_{R,\max}$,
\begin{equation}
\ol{u_{\ell}\Rpko v_{\ell}} = \ol{u_{\ell}\Rdko v_{\ell}} + \ol{u_{\ell}P_{\ell,\varphi,R}^{\free}(-k^2)v_{\ell}},
\end{equation}
and one verifies that
\begin{equation}\lb{3.29c}
\ol{u_{\ell}P_{\ell,\varphi,R}^{\free}(-k^2)v_{\ell}} = \sum_{m,n=1}^2\Big[K_{\ell,\varphi,R}^{\free}(-k^2)^{-1}\Big]_{m,n}\big\langle v_{\ell}\psi_{\ell,n}^{\free}(-k^2,\dott),\dott\big\rangle_{L^2((-\ell,\ell))}u_{\ell}\psi_{\ell,m}^{\free}(-k^2,\dott).
\end{equation}
To prove \eqref{3.29c} one observes that the operator on the right-hand side in \eqref{3.29c} is bounded on $L^2((-\ell,\ell))$ and that  $u_{\ell}P_{\ell,\varphi,R}^{\free}(-k^2)v_{\ell}$ coincides with this operator on the dense subspace $\dom{(v_{\ell})}$. 
 For $k>\max{\Big\{k_R^{\free},(1+N_R)^{1/2}\Big\}}$ one estimates:
\begin{align}
&\left\| \overline{u_\ell \Rpko v_\ell} \right\|_{\cB_1(L^2((-\ell,\ell)))}\no\\
&\quad \leq \left\|\overline{u_\ell\Rdko v_\ell}  \right\|_{\cB_1(L^2((-\ell,\ell)))} + \left\| \ol{u_\ell P_{\ell,\varphi,R}^{(0)}(-k^2) v_\ell} \right\|_{\cB_1(L^2((-\ell,\ell)))}.\lb{3.30c}
\end{align}
The first term on the right-hand side of the inequality in \eqref{3.30c} converges to zero as $k\to \infty$ by Lemma \ref{l3.3}.  Therefore, to establish \eqref{3.27c}, it suffices to show that
\begin{equation}\lb{3.31c}
\lim_{k \to \infty} \left\| \ol{u_\ell P_{\ell,\varphi,R}^{(0)}(-k^2) v_\ell} \right\|_{\cB_1(L^2((-\ell,\ell)))} = 0.
\end{equation}
To this end, \eqref{3.29c}, \eqref{2.42}, and Proposition \ref{pA.1} imply
\begin{equation}\lb{3.32c}
\begin{split}
&\left\| \ol{u_\ell P_{\ell,\varphi,R}^{(0)}(-k^2) v_\ell} \right\|_{\cB_1(L^2((-\ell,\ell)))}\\
&\,\;\underset{k\to \infty}{\leq}O(1/k)\sum_{m,n=1}^2\big\|v_{\ell}\psi_{\ell,n}^{\free}(-k^2,\dott) \big\|_{L^2((-\ell,\ell))}\big\|u_{\ell}\psi_{\ell,m}^{\free}(-k^2,\dott) \big\|_{L^2((-\ell,\ell))}.
\end{split}
\end{equation}
One infers that for any $k>0$,
\begin{align}
&\big\|u_\ell \psi_{\ell,m}^{\free}(-k^2,\dott)\big\|_{L^2((-\ell,\ell))}^2\lb{3.33c}\\
&\quad = \big\|v_\ell \psi_{\ell,m}^{\free}(-k^2,\dott)\big\|_{L^2((-\ell,\ell))}^2 = \int_{-\ell}^\ell \frac{1}{4}\left| \frac{\cosh{(kx)}}{\cosh{(k\ell)}} + (-1)^{m-1} \frac{\sinh{(kx)}}{\sinh{(k\ell)}}\right|^2 |V_{\ell}(x)|\, \mathrm dx,\quad m\in\{1,2\}.\no
\end{align}
The integrand on the right-hand side in \eqref{3.33c} is bounded a.e.~on $(-\ell,\ell)$ by $|V_{\ell}|\in L^1((-\ell,\ell))$ and since
\begin{equation}
\lim_{k\to \infty}\frac{1}{4}\left| \frac{\cosh{(kx)}}{\cosh{(k\ell)}} + (-1)^{m-1} \frac{\sinh{(kx)}}{\sinh{(k\ell)}}\right|^2 = 0,\quad x\in (-\ell,\ell),\, m\in\{1,2\},
\end{equation}
the dominated convergence theorem implies
\begin{equation}\lb{3.35c}
\lim_{k\to \infty}\big\|u_\ell \psi_{\ell,m}^{\free}(-k^2,\dott)\big\|_{L^2((-\ell,\ell))} = \lim_{k\to \infty}\big\|v_\ell \psi_{\ell,m}^{\free}(-k^2,\dott)\big\|_{L^2((-\ell,\ell))} = 0,\quad m\in \{1,2\}.
\end{equation}
The convergence statement in \eqref{3.31c} follows from \eqref{3.32c} and \eqref{3.35c}.
\end{proof}
%%%%%%%%%%%

We recall the following convergence results from \cite{GN12} for the operators $H_{\ell,D}^{\free}$, $\ell\in \bbN$.

%%%%%%%%%%%
\begin{lemma}[Lemmata 3.1 and 3.2 in \cite{GN12}]\lb{l3.2}
Assume Hypothesis \ref{h2.1}.  For each fixed $z\in \bbC\backslash[0,\infty)$, the following convergence results hold in $\cB_2(L^2(\bbR))$:
\begin{align}\lb{3.11}
\lim_{\ell \to \infty}\Big\|\Big[u_{\ell} \Rdzo \opl 0\Big] - u\Rrzo\Big\|_{\cB_2(L^2(\bbR))}&=0,\\
\lb{3.12}
\lim_{\ell \to \infty}\bigg\|\bigg[\overline{\Rdzo v_{\ell}}\opl 0 \bigg] - \overline{\Rrzo v} \bigg\|_{\cB_2(L^2(\bbR))}&=0,
\end{align}
and the following convergence result holds in $\cB_1(L^2(\bbR))$:
\begin{equation}\lb{3.13}
\lim_{\ell\to \infty}\bigg\|\bigg[\overline{u_{\ell}\Rdzo v_{\ell}}\opl 0 \bigg] - \overline{u\Rrzo v} \bigg\|_{\cB_1(L^2(\bbR))}=0.
\end{equation}
\end{lemma}
%%%%%%%%%%%

%%%%%%%%%%%
\begin{remark}
Statements analogous to \eqref{3.11} and \eqref{3.12} with the roles of $u_{\ell}$, $v_{\ell}$ and $u$, $v$ interchanged also hold.  That is, under the same assumptions in Lemma \ref{l3.2}, one has
\begin{align}\lb{3.11a}
\lim_{\ell \to \infty}\Big\|\Big[v_{\ell} \Rdzo \opl 0\Big] - v\Rrzo\Big\|_{\cB_2(L^2(\bbR))}&=0,\\
\lim_{\ell \to \infty}\bigg\|\bigg[\overline{\Rdzo u_{\ell}}\opl 0 \bigg] - \overline{\Rrzo u} \bigg\|_{\cB_2(L^2(\bbR))}&=0.
\end{align}
\hfill$\diamond$
\end{remark}
%%%%%%%%%%%

Using Theorem \ref{t2.6} and Lemma \ref{l3.2}, one establishes analogues of \eqref{3.11}--\eqref{3.13} for the operators $H_{\ell,\varphi,R}^{\free}$ with coupled boundary conditions.  Our proof of the analogue of \eqref{3.11} makes use of Gr\"umm's theorem, which we recall for completeness.

%%%%%%%%%%%
\begin{theorem}[Gr\"{u}mm's Theorem, \cite{Gr73}]
\label{grumm}
Let $p\in[1,\infty)$, $A\in\mathcal{B}_p(\hil)$, and suppose that $\{A_{\ell}\}_{\ell=1}^{\infty}\subseteq\mathcal{B}_p(\hil)$ with $\lim_{\ell\to\infty}{\|A_{\ell}-A\|_{\mathcal{B}_p(\hil)}}=0$. If $B\in\mathcal{B}(\hil)$, $\{\mathcal{B}_{\ell}\}_{\ell=1}^{\infty}\subseteq\mathcal{B}(\hil)$ with $\sup_{\ell\in\mathbb{N}}{\|B_{\ell}\|_{\mathcal{B}(\hil)}}<\infty$ and $\stlim_{\ell\to\infty}{B_{\ell}}=B$, then
\begin{equation}
\lim_{\ell\to\infty}{\|A_{\ell}B_{\ell}-AB\|_{\mathcal{B}_p(\hil)}}=\lim_{\ell\to\infty}{\|B_{\ell}A_{\ell}-BA\|_{\mathcal{B}_p(\hil)}}=0.
\end{equation}
\end{theorem}
%%%%%%%%%%%

The analogue of \eqref{3.11} and \eqref{3.12} for $H_{\ell,\varphi,R}^{\free}$ then reads as follows:

%%%%%%%%%%%
\begin{lemma}\lb{l3.8}
Assume Hypothesis \ref{h2.1} and let $k_{R,\max}\coloneqq \max\big\{k_R^{\free},(1+N_R)^{1/2}\big\}$.  For each fixed $ z \in \mathbb{C} \setminus [-k_{R,\max}^2,\infty)$, the following convergence results hold in $\cB_2(L^2(\bbR))$:
\begin{align}
\lim_{\ell \to \infty}\Big\|\Big[u_{\ell} \Rpzo \opl 0\Big] - u\Rrzo\Big\|_{\cB_2(L^2(\bbR))}&=0,\lb{3.39}\\
\lim_{\ell \to \infty}\bigg\|\bigg[\overline{\Rpzo v_{\ell}}\opl 0 \bigg] - \overline{\Rrzo v} \bigg\|_{\cB_2(L^2(\bbR))}&=0.\lb{3.40}
\end{align}
\end{lemma}
%%%%%%%%%%%
\begin{proof}
Let $k>k_{R,\max}$.  We first verify \eqref{3.39} for $z=-k^2$.  Using Theorem \ref{t2.6} to relate the resolvent operators of $H_{\ell,\varphi,R}^{\free}$ and $H_{\ell,D}^{\free}$, one obtains:
\begin{align}
&\Big\|\Big[u_{\ell} \Rpko \opl 0\Big] - u\Rrko\Big\|_{\cB_2(L^2(\bbR))}\lb{3.42}\\
&\quad \leq \Big\|\Big[u_{\ell} \Rdko \opl 0\Big] - u\Rrko\Big\|_{\cB_2(L^2(\bbR))}\no\\
&\qquad + \Big\|u_{\ell}P_{\ell,\varphi,R}^{\free}(-k^2)\opl 0\Big\|_{\cB_2(L^2(\bbR))},\quad \ell\in \bbN.\no
\end{align}
The first $\cB_2$-norm on the right-hand side of the inequality in \eqref{3.42} converges to zero as $\ell\to \infty$ by Lemma \ref{l3.2}.  Thus, it suffices to show
\begin{equation}
\lim_{\ell\to \infty}\Big\|u_{\ell}P_{\ell,\varphi,R}^{\free}(-k^2)\opl 0\Big\|_{\cB_2(L^2(\bbR))}=0.
\end{equation}
To this end, \eqref{2.28b}, \eqref{2.31}, and Proposition \ref{pA.1} imply
\begin{align}
&\Big\|u_{\ell}P_{\ell,\varphi,R}^{\free}(-k^2)\opl 0\Big\|_{\cB_2(L^2(\bbR))}\no\\
&\quad= \Big\|u_{\ell}P_{\ell,\varphi,R}^{\free}(-k^2)\Big\|_{\cB_2(L^2((-\ell,\ell)))}\no\\
&\quad= \Bigg\|\sum_{m,n=1}^2 \Big[K_{\ell,\varphi,R}^{\free}(-k^2)^{-1}\Big]_{m,n}\big\langle \psi_{\ell,n}^{\free}(-k^2,\dott),\dott\big\rangle_{L^2((-\ell,\ell))}u_{\ell}\psi_{\ell,m}^{\free}(-k^2,\dott) \Bigg\|_{\cB_2(L^2((-\ell,\ell)))}\no\\
&\,\;\underset{\ell\to \infty}{\leq}O(1)\cdot\sum_{m,n=1}^2\big\|\psi_{\ell,n}^{\free}(-k^2,\dott)\big\|_{L^2((-\ell,\ell))}\big\|u_{\ell}\psi_{\ell,m}^{\free}(-k^2,\dott)\big\|_{L^2((-\ell,\ell))}\no\\
&\,\;\underset{\ell\to \infty}{\leq}O(1)\cdot\sum_{m=1}^2\big\|u_{\ell}\psi_{\ell,m}^{\free}(-k^2,\dott)\big\|_{L^2((-\ell,\ell))},\quad \ell\in \bbN.\lb{3.45}
\end{align}
The first inequality in \eqref{3.45} follows from \eqref{2.31} while the second inequality follows from \eqref{3.18b}.  The $L^2$-norms in the sum after the final inequality in \eqref{3.45} converge to zero as $\ell\to \infty$.  Indeed, writing for $m\in \{1,2\}$ and $\ell\in \bbN$,
\begin{equation}
\begin{split}
\big\|u_{\ell}\psi_{\ell,m}^{\free}(-k^2,\dott)\big\|_{L^2((-\ell,\ell))}^2 = \frac{1}{4}\int_{-\infty}^{\infty}\bigg|\frac{\cosh{(kx)}}{\cosh{(k\ell)}} + (-1)^{m-1}\frac{\sinh{(kx)}}{\sinh{(k\ell)}}\bigg|^2\chi_{(-\ell,\ell)}(x)|V(x)|\,\textrm{d}x,
\end{split}
\end{equation}
one notes that for $m\in \{1,2\}$ and $\ell\in \bbN$,
\begin{equation}\lb{3.48}
\frac{1}{4}\bigg|\frac{\cosh{(kx)}}{\cosh{(k\ell)}} + (-1)^{m-1}\frac{\sinh{(kx)}}{\sinh{(k\ell)}}\bigg|^2\chi_{(-\ell,\ell)}(x)|V(x)| \leq |V(x)|\,\text{ for a.e.~$x\in \bbR$}.
\end{equation}
Since $V\in L^1(\bbR)$ and the left-hand side of the inequality in \eqref{3.48} converges to zero pointwise a.e.~on $\bbR$, the dominated convergence theorem implies
\begin{equation}\lb{3.49}
\lim_{\ell\to \infty} \big\|u_{\ell}\psi_{\ell,m}^{\free}(-k^2,\dott)\big\|_{L^2((-\ell,\ell))} =0,\quad m\in \{1,2\}.
\end{equation}
Thus, by combining \eqref{3.45} and \eqref{3.49}, one obtains for any $k>k_{R,\max}$,
\begin{equation}\lb{3.46z}
\lim_{\ell \to \infty}\Big\|\Big[u_{\ell} \Rpko \opl 0\Big] - u\Rrko\Big\|_{\cB_2(L^2(\bbR))}=0.
\end{equation}

To prove \eqref{3.39} in full generality, let $z\in \bbC\backslash [-k_{R,\max}^2,\infty)$.  Fix any $k>k_{R,\max}$.  By the first resolvent identity applied to $H_{\ell,\varphi,R}^{\free}$,
\begin{align}
&u_{\ell}\Rpzo \opl 0\lb{3.47z}\\
&\quad = \Big[u_{\ell}\Rpko \opl 0\Big]\no\\
&\qquad + (z+k^2)\Big[u_{\ell}\Rpko \opl 0\Big]\Big[\Rpzo \opl -z^{-1}I_{L^2(\bbR\backslash(-\ell,\ell))}\Big]\no\\
&\quad = \Big[u_{\ell}\Rpko \opl 0\Big]\no\\
&\qquad + (z+k^2)\Big[u_{\ell}\Rpko \opl 0\Big]\Big(\Big[ H_{\ell,\varphi,R}^{(0)} \opl 0 \Big] - z I_{L^2(\bbR)} \Big)^{-1},\quad \ell\in \bbN.\no
\end{align}
In light of the estimate
\begin{equation}
\begin{split}
\Big\|\Big(\Big[ H_{\ell,\varphi,R}^{(0)} \opl 0 \Big] - z I_{L^2(\bbR)} \Big)^{-1}\Big\|_{\cB(L^2(\bbR))}& \leq \Big\|\Rpzo\Big\|_{\cB(L^2((-\ell,\ell)))} + |z|^{-1}\\
&\leq \dist{\big(z,[-k_{R,\max}^2,\infty)\big)} + |z|^{-1},\quad \ell\in \bbN,
\end{split}
\end{equation}
where $\dist{(\zeta,\Omega)}$ denotes the distance between the point $\zeta\in \bbC$ and the set $\Omega\subseteq \bbC$, one concludes
\begin{equation}
\sup_{\ell\in \bbN} \Big\|\Big(\Big[ H_{\ell,\varphi,R}^{(0)} \opl 0 \Big] - z I_{L^2(\bbR)} \Big)^{-1}\Big\|_{\cB(L^2(\bbR))} < \infty.
\end{equation}
Therefore, by \eqref{3.46z} and Lemma \ref{l3.2b} combined with Gr\"umm's theorem (to treat the second term after the second equality in \eqref{3.47z}), the right-hand side in \eqref{3.47z} converges in $\cB_2\big(L^2(\bbR)\big)$ as $\ell\to \infty$ and
\begin{align}
&\lim_{\ell\to \infty}\Big[u_{\ell}\Rpzo \opl 0\Big]\lb{3.48z}\\
&\quad = u\Rrko + (z+k^2)u\Rrko \Rrzo\no\\
&\quad = u\Rrzo\, \text{ in $\cB_2\big(L^2(\bbR)\big)$},\no
\end{align}
which establishes \eqref{3.39}.  The final equality in \eqref{3.48z} is the result of another application of the first resolvent identity, this time for $H^{\free}$.

Finally, to prove \eqref{3.40}, note that the same argument employed above to prove \eqref{3.39} shows \textit{mutatis mutandis} (using \eqref{3.11a}) that
\begin{equation}\lb{3.49z}
\lim_{\ell \to \infty}\Big\|\Big[v_{\ell} \Rpzo \opl 0\Big] - v\Rrzo\Big\|_{\cB_2(L^2(\bbR))}=0
\end{equation}
for $z\in \bbC\backslash [-k_{R,\max}^2,\infty)$.  Therefore, for $z\in \bbC\backslash[-k_{R,\max}^2,\infty)$, taking adjoints one obtains
\begin{align}
&\bigg\|\bigg[\overline{\Rpzo v_{\ell}}\opl 0 \bigg] - \overline{\Rrzo v} \bigg\|_{\cB_2(L^2(\bbR))}\\
&\quad =\Big\|\Big[\Rpzo v_{\ell}\opl 0 \Big]^{\ast \ast} - \Big[\Rrzo v\Big]^{\ast\ast} \Big\|_{\cB_2(L^2(\bbR))}\no\\
&\quad =\Big\|\Big[v_{\ell}\Rpzbo \opl 0 \Big]^{\ast } - \Big[v\Rrzbo \Big]^{\ast} \Big\|_{\cB_2(L^2(\bbR))}\no\\
&\quad =\Big\|\Big[v_{\ell}\Rpzbo \opl 0 \Big] - v\Rrzbo \Big\|_{\cB_2(L^2(\bbR))},\quad \ell\in \bbN,\no
\end{align}
which converges to zero as $\ell\to \infty$ by \eqref{3.49z}.  The claim in \eqref{3.40} follows.
\end{proof}
%%%%%%%%%%%

Our proof of the analogue of \eqref{3.13} for $H_{\ell,\varphi,R}^{\free}$ relies on the following classical result for the limit of a sequence that is the term-wise product of operators belonging to appropriate trace ideals.

%%%%%%%%%%%
\begin{lemma}\lb{l3.9}
Let $p,q,r\in[1,\infty)$ with $p^{-1}+q^{-1}=r^{-1}$.  If $\{A_\ell\}_{\ell=1}^{\infty}\subset \cB_p(\cH)$, $\{B_\ell\}_{\ell=1}^{\infty}\subset \cB_q(\cH)$, $A\in \cB_p(\cH)$, and $B\in \cB_q(\cH)$ with
\begin{equation}
\lim_{\ell\to \infty}\|A_\ell-A\|_{\cB_p(\cH)}=0\quad \text{and}\quad \lim_{\ell\to \infty}\|B_\ell-B\|_{\cB_q(\cH)}=0,
\end{equation}
then
\begin{equation}
\lim_{\ell\to \infty}\|A_\ell B_\ell-AB\|_{\cB_r(\cH)}=0.
\end{equation}
\end{lemma}
%%%%%%%%%%%
The proof of Lemma \ref{l3.9} is an application of H\"older's inequality for the trace ideals (cf., e.g., \cite[Theorem 2.8]{Si05}).  Lemma \ref{l3.9} then combines with Lemma \ref{l3.8} to yield the following analogue of \eqref{3.13} for $H_{\ell,\varphi,R}^{\free}$.

%%%%%%%%%%%
\begin{lemma}\lb{l3.10}
Assume Hypothesis \ref{h2.1} and let $k_{R,\max}\coloneqq \max\big\{k_R^{\free},(1+N_R)^{1/2}\big\}$.  For each fixed $ z \in \mathbb{C} \setminus [-k_{R,\max}^2,\infty)$, the following convergence result holds in $\cB_1(L^2(\bbR))$:
\begin{equation}\lb{3.59}
\lim_{\ell\to \infty}\bigg\|\bigg[\overline{u_{\ell}\Rpzo v_{\ell}}\opl 0 \bigg] - \overline{u\Rrzo v} \bigg\|_{\cB_1(L^2(\bbR))}=0.
\end{equation}
\end{lemma}
%%%%%%%%%%%
\begin{proof}
Let $k>k_{R,\max}$.  We first verify \eqref{3.59} for $z=-k^2$.  Using Theorem \ref{t2.6} to relate the resolvent operators of $H_{\ell,\varphi,R}^{\free}$ and $H_{\ell,D}^{\free}$, one obtains:
\begin{align}
&\bigg\|\bigg[\ol{u_{\ell} \Rpko v_{\ell}} \opl 0\bigg] - \ol{u\Rrko v}\bigg\|_{\cB_1(L^2(\bbR))}\lb{3.60}\\
&\quad \leq \bigg\|\bigg[\ol{u_{\ell} \Rdko v} \opl 0\bigg] - \ol{u\Rrko v}\bigg\|_{\cB_1(L^2(\bbR))}\no\\
&\qquad + \bigg\|\ol{u_{\ell}P_{\ell,\varphi,R}^{\free}(-k^2)v_{\ell}}\opl 0\bigg\|_{\cB_1(L^2(\bbR))},\quad \ell\in \bbN.\no
\end{align}
The first $\cB_1$-norm on the right-hand side of the inequality in \eqref{3.60} converges to zero as $\ell\to \infty$ by Lemma \ref{l3.2}.  Invoking \eqref{3.29c}, \eqref{2.31}, and Proposition \ref{pA.1}, one obtains
\begin{align}
&\bigg\|\ol{u_{\ell}P_{\ell,\varphi,R}^{\free}(-k^2)v_{\ell}}\opl 0\bigg\|_{\cB_1(L^2(\bbR))}\lb{3.61}\\
&\quad = \bigg\|\ol{u_{\ell}P_{\ell,\varphi,R}^{\free}(-k^2)v_{\ell}}\bigg\|_{\cB_1(L^2((-\ell,\ell)))}\no\\
&\,\; \underset{\ell\to \infty}{\leq} O(1)\cdot\sum_{m,n=1}^2\big\|v_{\ell}\psi_{\ell,n}^{\free}(-k^2,\dott)\big\|_{L^2((-\ell,\ell))}\big\|u_{\ell}\psi_{\ell,m}^{\free}(-k^2,\dott)\big\|_{L^2((-\ell,\ell))},\quad \ell\in \bbN.\no
\end{align}
Since $|u_{\ell}|=|v_{\ell}|$ a.e.~on $(-\ell,\ell)$, \eqref{3.49} implies that each factor under the sum in \eqref{3.61}, and hence the sum itself, converges to zero as $\ell\to \infty$.  Thus, taking $\ell\to \infty$ throughout \eqref{3.60}, one concludes:
\begin{align}
\lim_{\ell\to \infty}\bigg\|\bigg[\ol{u_{\ell} \Rpko v_{\ell}} \opl 0\bigg] - \ol{u\Rrko v}\bigg\|_{\cB_1(L^2(\bbR))} = 0.\lb{3.62}
\end{align}
To prove \eqref{3.59} in full generality, let $z\in \bbC\backslash[-k_{R,\max}^2,\infty)$.  Fix any $k>k_{R,\max}$.  By the first resolvent identity applied to $H_{\ell,\varphi,R}^{\free}$,
\begin{align}
&\ol{u_{\ell}\Rpzo v_{\ell}} \opl 0\lb{3.63}\\
&\quad = \bigg[\ol{u_{\ell}\Rpko v_{\ell}} \opl 0\bigg]\no\\
&\qquad + (z+k^2)\bigg[u_{\ell}\Rpko \opl 0\bigg]\bigg[\ol{\Rpzo v_{\ell}} \opl 0\bigg],\quad \ell\in \bbN.\no
\end{align}
Taking $\ell\to \infty$ throughout \eqref{3.63} and invoking \eqref{3.62} to treat the first term on the right-hand side in \eqref{3.63} and using Lemmata \ref{l3.8} and \ref{l3.9} to treat the second term on the right-hand side in \eqref{3.63}, one obtains
\begin{align}
&\lim_{\ell\to \infty} \ol{u_{\ell}\Rpzo v_{\ell}} \opl 0\lb{3.65}\\
&\quad = \ol{u\Rrko v} + (z+k^2)u\Rrko\ol{\Rrzo v}\no\\
&\quad = \ol{u\Rrzo v}\, \text{ in $\cB_1\big(L^2(\bbR)\big)$},\no
\end{align}
and \eqref{3.59} follows.  The final equality in \eqref{3.65} is the result of another application of the first resolvent identity, this time for $H^{\free}$.
\end{proof}
%%%%%%%%%%%

%%%%%%%%%%%

%%%%%%%%%%%

%%%%%%%%%%%%%%%%%%%%%%%%%%%%%%%
%%%%%%%%%%%%%%%%%%%%%%%%%%%%%%%
\section{Convergence Results for Spectral Shift Functions} \lb{s4}
%%%%%%%%%%%%%%%%%%%%%%%%%%%%%%%
%%%%%%%%%%%%%%%%%%%%%%%%%%%%%%%

Hypothesis \ref{h2.1} guarantees that $H^{\free}$ and $H$ are resolvent comparable; that is,
\begin{equation}\lb{4.1}
\Big[\big(H-zI_{L^2(\bbR)}\big)^{-1} -\Rrzo\Big] \in \cB_1\big(L^2(\bbR)\big),\quad z\in \bbC\backslash\bbR.
\end{equation}
% The trace class containment \eqref{4.1} follows from the Kato resolvent equation:
% \begin{align}
% &\big(H-zI_{L^2(\bbR)}\big)^{-1} -\Rrzo\\
% &\quad = \ol{\Rrzo v_{\ell}}\bigg[I_{L^2(\bbR)}+\ol{u_{\ell}\Rrzo v_{\ell}}\bigg]^{-1}u_{\ell}\Rrzo,\quad z\in\rho(H),\no
% \end{align}
% combined with Lemma \ref{l2.4a}.
In a similar vein, for each $\ell\in\bbN$,
\begin{equation}\lb{4.3}
\Big[\big(H_{\ell,D}-zI_{L^2((-\ell,\ell))}\big)^{-1} - \Rdzo\Big]\in \cB_1\big(L^2((-\ell,\ell))\big),\quad z\in\bbC\backslash\bbR.
\end{equation}
Moreover, Theorem \ref{t2.6} and \eqref{4.3} imply $H_{\ell,\varphi,R}$ and $H_{\ell,\varphi,R}^{\free}$ are resolvent comparable for every $\ell\in \bbN$:
\begin{equation}\lb{4.4}
\Big[\big(H_{\ell,\varphi,R}-zI_{L^2((-\ell,\ell))}\big)^{-1} - \Rpzo\Big] \in \cB_1\big(L^2((-\ell,\ell))\big),\quad z\in \bbC\backslash\bbR.
\end{equation}
By \eqref{4.1} and \eqref{4.4} there exist unique real-valued (a.e.) spectral shift functions $\xi\big(\dott;H,H^{\free}\big)$ and $\xi\big(\dott;H_{\ell,\varphi,R},H_{\ell,\varphi,R}^{\free}\big)$ for the pairs $\big(H,H^{\free}\big)$ and $\big(H_{\ell,\varphi,R},H_{\ell,\varphi,R}^{\free}\big)$, $\ell\in \bbN$, respectively, that satisfy
\begin{equation}
\int_{-\infty}^{\infty} \frac{\big|\xi\big(\lambda;H,H^{\free}\big)\big|}{1+\lambda^2}\,\textrm{d}\lambda<\infty,\quad\text{and}\quad\xi\big(\lambda;H,H^{\free}\big)=0,\quad\lambda<\min\sigma(H),
\end{equation}
and
\begin{equation}
\begin{split}
&\int_{-\infty}^{\infty}\frac{\big|\xi\big(\lambda;H_{\ell,\varphi,R},H_{\ell,\varphi,R}^{\free}\big)\big|}{1+\lambda^2}\,\textrm{d}\lambda<\infty,\\ &\xi\big(\lambda;H_{\lambda,\varphi,R},H_{\ell,\varphi,R}^{\free}\big)=0,\quad\lambda<\min\big[\sigma\big(H_{\ell,\varphi,R}\big)\cup\sigma\big(H_{\ell,\varphi,R}^{\free}\big)\big],\quad \ell\in \bbN.
\end{split}
\end{equation}
and for which Krein's trace formul\ae~hold:  For any $f\in \mathfrak{F}(\bbR)$ (cf.~\eqref{A.24}--\eqref{A.25}),
\begin{align}
\tr_{L^2(\bbR)}\big(f(H)-f\big(H^{\free}\big)\big) &= \int_{-\infty}^{\infty}f'(\lambda)\xi\big(\lambda;H,H^{\free}\big)\,\textrm{d}\lambda,\\
\tr_{L^2((-\ell,\ell))}\big(f\big(H_{\ell,\varphi,R}\big) - f\big(H_{\ell,\varphi,R}^{\free}\big)\big) &= \int_{-\infty}^{\infty}f'(\lambda)\xi\big(\lambda;H_{\ell,\varphi,R},H_{\ell,\varphi,R}^{\free}\big)\,\textrm{d}\lambda,\quad \ell\in \bbN.
\end{align}
The convergence properties in Lemmata \ref{l3.1}, \ref{l3.2b}, \ref{l3.4}, \ref{l3.8}, and \ref{l3.10} imply the following weak convergence property of $\xi\big(\dott;H_{\ell,\varphi,R}\big)$ to $\xi\big(\dott;H,H^{\free}\big)$ as $\ell\to \infty$.

%%%%%%%%%%%
\begin{theorem}
If Hypothesis \ref{h2.1} holds, then
\begin{equation}\lb{4.8}
\lim_{\ell\rightarrow \infty}\int_{-\infty}^{\infty} 
\frac{\xi\big(\lambda; H_{\ell,\varphi,R},H_{\ell,\varphi,R}^{\free}\big)}{1+\lambda^2} \, f(\lambda) \, \mathrm{d}\lambda
= \int_{-\infty}^{\infty} \frac{\xi\big(\lambda;H,H^{\free}\big)}{1+\lambda^2} \, f(\lambda)\, \mathrm{d}\lambda,  
\quad f\in C_b(\bbR).
\end{equation} 
\end{theorem}
%%%%%%%%%%%
\begin{proof}
It suffices to verify that conditions $(i)$--$(ix)$ in Hypothesis \ref{hB.1} hold with the identifications
\begin{equation}
\begin{split}
&T^{\free} = H^{\free},\hspace*{.78cm} T=H,\hspace*{1.3cm} V_1=u,\hspace*{.72cm} V_2=v,\\ &T^{\free}_{\ell}=H_{\ell,\varphi,R}^{\free},\quad T_{\ell}=H_{\ell,\varphi,R},\quad V_{1,\ell}=u_{\ell},\quad V_{2,\ell}=v_{\ell},\quad \ell\in \bbN.
\end{split}
\end{equation}
The conclusion of the theorem then follows from Theorem \ref{tB.13}.  Condition $(i)$ holds by construction and Lemmata \ref{l2.3} and \ref{l2.4} imply that $(ii)$ holds.  Since $u,v\in L^2(\bbR)$ are real-valued a.e.~and functions in $\dom{(|H^{\free}|^{1/2})}=H^1(\bbR)$ are bounded, one infers that the operators of multiplication by $u$ and $v$ are self-adjoint, hence closed, and that \eqref{B.2} an \eqref{B.62} hold.  Similarly, one verifies that \eqref{B.3} and \eqref{B.63} hold.  Therefore, condition $(iii)$ is satisfied.  Condition $(iv)$ holds by Lemmata \ref{l2.4a}, \ref{l2.10}, and \ref{l3.4}.  Condition $(v)$ holds by Lemma \ref{l3.2b}, and condition $(vi)$ holds by Lemmata \ref{l3.8} and \ref{l3.10}.  Condition $(vii)$ holds since $V_1=u$, $V_2=v$, $V_{1,\ell}=u_{\ell}$, and $V_{2,\ell}=v_{\ell}$ are operators of multiplication by real-valued functions, and condition $(viii)$ holds by Remarks \ref{r2.5} and \ref{r2.12}.  Finally, the explicit structures of the sesquilinear forms $\mathfrak{Q}$ and $\mathfrak{Q}_{\ell,\varphi,R}$ (cf.~\eqref{2.3} and \eqref{2.12}) imply
\begin{equation}
\mathfrak{Q}=\mathfrak{Q}^{\free}+\mathfrak{Q}_V,\quad \mathfrak{Q}_{\ell,\varphi,R}=\mathfrak{Q}_{\ell,\varphi,R}^{\free}+\mathfrak{Q}_{V_{\ell}},\quad \ell\in \bbN,
\end{equation}
where $\mathfrak{Q}_V$ and $\mathfrak{Q}_{V_{\ell}}$ denote the sesquilinear forms uniquely associated to the self-adjoint operators of multiplication by $V$ and $V_{\ell}$ in $L^2(\bbR)$ and $L^2((-\ell,\ell))$, $\ell\in\bbN$, respectively.  In particular, $H$ is the form sum of $H^{\free}$ with $V$ and $H_{\ell,\varphi,R}$ is the form sum of $H_{\ell,\varphi,R}^{\free}$ with $V_{\ell}$, $\ell\in \bbN$.  Hence, condition $(ix)$ holds.
\end{proof}
%%%%%%%%%%%

Invoking the abstract Corollaries \ref{cB.11}, \ref{cB.14}, and \ref{cB.15}, one obtains the following additional results for convergence of the spectral shift functions.

%%%%%%%%%%
\begin{corollary}
If Hypothesis \ref{h2.1} holds, then 
\begin{equation}
\lim_{\ell\to \infty} \int_{-\infty}^{\infty}\xi\big(\lambda;H_{\ell,\varphi,R},H_{\ell,\varphi,R}^{\free}\big) \, g(\lambda) \, \mathrm{d}\lambda
= \int_{-\infty}^{\infty}\xi\big(\lambda;H,H^{\free}\big) \, g(\lambda)\, \mathrm{d}\lambda,\quad g \in C_c(\bbR).
\end{equation} 
\end{corollary}
%%%%%%%%%%

%%%%%%%%%%%
\begin{corollary}
If Hypothesis \ref{h2.1} holds, then the convergence in \eqref{4.8} holds for any bounded Borel measurable function that is continuous almost everywhere with respect to Lebesgue measure on $\bbR$.  In particular, 
\begin{equation}
\lim_{\ell\to \infty}\int_{S}\frac{\xi\big(\lambda;H_{\ell,\varphi,R},H_{\ell,\varphi,R}^{\free}\big)}{1+\lambda^2}\,\mathrm{d}\lambda = \int_{S}\frac{\xi\big(\lambda;H,H^{\free}\big)}{1+\lambda^2}\,\mathrm{d}\lambda
\end{equation}
holds for any set $S\subseteq\bbR$ that is boundaryless with respect to Lebesgue 
measure $($i.e., any set $S\subseteq\bbR$ for which the boundary of $S$ has Lebesgue measure 
equal to zero$)$.
\end{corollary}
%%%%%%%%%%%

%%%%%%%%%%%
\begin{corollary}
If Hypothesis \ref{hB.1} holds, then for every bounded Borel measurable function $g$ that is compactly supported and Lebesgue almost everywhere continuous on $\bbR$,
\begin{equation}
\lim_{\ell\rightarrow \infty}\int_{-\infty}^{\infty} \xi\big(\lambda; H_{\ell,\varphi,R},H_{\ell,\varphi,R}^{\free}\big)  \, g(\lambda) \, \mathrm{d}\lambda
= \int_{-\infty}^{\infty} \xi\big(\lambda;H,H^{\free}\big) \, g(\lambda)\, \mathrm{d}\lambda.
\end{equation}
\end{corollary}
%%%%%%%%%%%

%%%%%%%%%%%%%%%%%%%%%%%%%%%%%%%%%%%%%
%%%%%%%%%%%%%% appendices %%%%%%%%%%%%%%%%%
\appendix
%%%%%%%%%%%%%%% Appendix A %%%%%%%%%%%%%%%%
\section{Criteria for Vague and Weak Convergence of Spectral Shift Functions} \lb{sA}
\renewcommand{\theequation}{A.\arabic{equation}}
\renewcommand{\thetheorem}{A.\arabic{theorem}}
\setcounter{theorem}{0} \setcounter{equation}{0}

We recall the criteria for convergence of sequences of spectral shift functions introduced in \cite{GN12b} suitably tailored to the context in which they are applied in this paper; namely, for pairs of self-adjoint operators acting in the Hilbert spaces $L^2((-\ell,\ell))$ and $L^2(\bbR)$.

%%%%%%%%%%%%
\begin{hypothesis}[Hypothesis 3.1 in \cite{GN12b}] \lb{hB.1} Set $\cH\coloneqq L^2(\bbR)$. \\
$(i)$ For each $\ell\in \bbN$, decompose $\cH$ according to 
\begin{equation}
L^2(\bbR) = L^2((-\ell,\ell))\oplus_{\ell}L^2(\bbR\backslash(-\ell,\ell)),
\end{equation} 
and define $\cH_{\ell}:=L^2((-\ell,\ell))$ and $\cH_{\ell}^c=L^2(\bbR\backslash(-\ell,\ell))$.\\
$(ii)$ Let $T^{(0)}$ be a self-adjoint operator in $\cH$, and for each $\ell \in \bbN$, 
let $T_{\ell}^{(0)}$ be self-adjoint operators in $\cH_{\ell}$. In addition, suppose that 
$T^{(0)}$ is lower semibounded in $\cH$, and that for each $\ell\in\bbN$, $T_{\ell}^{(0)}$ is lower semibounded in $\cH_{\ell}$. \\
$(iii)$ Suppose that $V_1$, and $V_2$ are closed operators in $\cH$, and for each 
$\ell \in\bbN$, assume that $V_{1,\ell}$, and $V_{2,\ell}$ are closed operators in $\cH_{\ell}$ such that 
\begin{align}
& \dom{(V_1)} \cap \dom{(V_2)} \supseteq \dom{\big(|T^{(0)}|^{1/2}\big)},    \lb{B.2} \\
& \dom{(V_{1,\ell})} \cap \dom{(V_{2,\ell})} \supseteq \dom{\big(|T_{\ell}^{(0)}|^{1/2}\big)}, \quad \ell\in\bbN,   \lb{B.3}
\end{align} 
where
\begin{equation}
V = V_1^* V_2 \, \text{ is a self-adjoint operator in $\cH$},    \lb{B.62}
\end{equation}
and for each $\ell\in\bbN$,
\begin{equation}
V_\ell = V_{1,\ell}^* V_{2,\ell} \, \text{ is a self-adjoint operator in $\cH_\ell$}.    \lb{B.63}
\end{equation}
$(iv)$  Suppose
\begin{align}
& \ol{V_2\big(T^{(0)} - z I_{\cH}\big)^{-1}V_1^*}, \, \ol{V_{2,\ell}\big(T_{\ell}^{(0)} - z I_{\cH_\ell}\big)^{-1}V_{1,\ell}^*} 
\oplus_\ell 0 \in \cB_1(\cH), \quad \ell \in \bbN,    \lb{B.4} \\
& V_2\big(T^{(0)} - z I_{\cH}\big)^{-1}, \, V_{2,\ell}\big(T_{\ell}^{(0)} -z I_{\cH_\ell}\big)^{-1}\oplus_\ell 0 \in \cB_{2}(\cH), 
\quad \ell\in\bbN,    \lb {B.5} \\
& \ol{\big(T^{(0)} - z I_{\cH}\big)^{-1}V_1^*}, \, \ol{\big(T_{\ell}^{(0)} - z I_{\cH_\ell}\big)^{-1}V_{1,\ell}^*}\oplus_\ell 0 
\in \cB_{2}(\cH), \quad \ell\in\bbN,    \lb{B.6}
\end{align}
for some $($and hence for all\,$)$ $z\in \bbC\backslash \bbR$. In addition, assume that 
\begin{align}
\begin{split} 
& \lim_{z \downarrow - \infty} \Big\|\,\ol{V_2 \big(T^{(0)} - z I_{\cH}\big)^{-1}V_1^*}\,\Big\|_{\cB_1(\cH)} =0, 
\lb{B.6a} \\ 
& \lim_{z \downarrow - \infty} \Big\|\,\ol{V_{2,\ell}\big(T_{\ell}^{(0)} - z I_{\cH_\ell}\big)^{-1}V_{1,\ell}^*} 
\oplus_\ell 0\,\Big\|_{\cB_1(\cH)} =0, \quad \ell \in\bbN.  
\end{split}
\end{align}
$(v)$ Assume that for some $($and hence for all\,$)$ $z\in \bbC\backslash \bbR$, 
\begin{equation}
\stlim_{\ell \to \infty}\bigg[\big(T_{\ell}^{(0)} - z I_{\cH_\ell}\big)^{-1}\oplus_\ell \frac{-1}{z} I_{\cH_{\ell}^c}\bigg] 
= \big(T^{(0)} - z I_{\cH}\big)^{-1}.    \lb{B.8}
\end{equation}
$(vi)$ Suppose that for some $($and hence for all\,$)$ $z\in \bbC\backslash \bbR$, 
\begin{align}
& \lim_{\ell \to \infty}\bigg\|\bigg[\,\ol{V_{2,\ell}\big(T_{\ell}^{(0)} - z I_{\cH_\ell}\big)^{-1}V_{1,\ell}^*}\oplus_\ell 0\bigg] 
- \ol{V_2 \big(T^{(0)} - z I_{\cH}\big)^{-1}V_1^*}\,\bigg\|_{\cB_1(\cH)} =0,    \lb{B.9} \\ 
& \lim_{\ell \to \infty} \Big\|\Big[V_{2,\ell}\big(T_{\ell}^{(0)} - z I_{\cH_\ell}\big)^{-1}\oplus_\ell 0\Big] - 
V_2 \big(T^{(0)} - z I_{\cH}\big)^{-1}\Big\|_{\cB_{2}(\cH)} =0,     \lb{B.10} \\ 
& \lim_{\ell \to \infty} \bigg\|\bigg[\,\ol{\big(T_{\ell}^{(0)} - z I_{\cH_\ell})^{-1}V_{1,\ell}^*}\oplus_\ell 0\bigg]
- \ol{\big(T^{(0)}-zI_{\cH}\big)^{-1}V_1^*}\,\bigg\|_{\cB_{2}(\cH)} =0.    \lb{B.11}
\end{align}
$(vii)$ Suppose that 
\begin{align}
\begin{split} 
& \langle V_2f,V_1g\rangle_{\cH}=\langle V_1f,V_2g\rangle_{\cH}, \quad f,g\in\dom{(V_1)}\cap\dom{(V_2)},   \lb{B.12} \\
& \langle V_{2,\ell}f,V_{1,\ell}g\rangle_{\cH}=\langle V_{1,\ell}f,V_{2,\ell}g\rangle_{\cH}, \quad f,g\in\dom{(V_{1,\ell})}
\cap\dom{(V_{2,\ell})}, \quad \ell \in\bbN. 
\end{split}
\end{align}
$(viii)$ Decomposing $V, V_\ell$, $\ell\in\bbN$, into their positive and negative parts, 
\begin{equation}
V_{\pm} = [|V| \pm V]/2, \quad V_{\ell,\pm} = [|V_\ell| \pm V_\ell]/2, \quad \ell\in\bbN,  
\lb{B.64}
\end{equation}
$V_{\pm}$ are assumed to be infinitesimally form bounded with respect to $T^{(0)}$, 
and for each $\ell\in\bbN$, $V_{\ell,\pm}$ are assumed to be infinitesimally form bounded 
with respect to $T_{\ell}^{(0)}$.\\
$(ix)$  Suppose that $T$ and $T_{\ell}$, $\ell\in \bbN$ are the quadratic forms sums of $T^{\free}$ and $T_{\ell}^{\free}$ with $V_1^*V_2$ and $V_{1,\ell}^*V_{2,\ell}$, respectively; that is,
\begin{equation}
T=T^{\free}+_\mathfrak{q}V_1^*V_2,\quad T_{\ell}=T_{\ell}^{\free}+_{\mathfrak{q}}V_{1,\ell}^*V_{2,\ell},\quad\ell\in \bbN.
\end{equation}
\end{hypothesis}
%%%%%%%%%%%%

% $(viii)$ $T$ and $T_{\ell}$ are equal to the quadratic form sums of $T^{(0)}$ with $V_1^*V_2$ and $T_{\ell}^{(0)}$ with $V_{1,\ell}^*V_{2,\ell}$, respectively:
% \begin{equation}
% \begin{split}
% T &= T^{(0)} +_q V_1^*V_2,\\
% T_{\ell} &= T_{\ell}^{(0)} +_q V_{1,\ell}^*V_{2,\ell},\quad \ell\in \bbN.
% \end{split}
% \end{equation}

Assuming Hypothesis \ref{hB.1}, the pairs $\big(T, T^{(0)}\big)$ and $\big(T_\ell,T_\ell^{(0)}\big)$, $\ell\in \bbN$, are resolvent comparable in the sense that
\begin{equation}
\big[\big(T-zI_{\cH}\big)^{-1} - \big(T^{(0)}-zI_{\cH}\big)^{-1}\big]\in \cB_1(\cH),\quad z\in \bbC\backslash \bbR,
\end{equation}
and 
\begin{equation}
\big[\big(T_\ell-zI_{\cH}\big)^{-1} - \big(T_{\ell}^{(0)}-zI_{\cH}\big)^{-1}\big]\in \cB_1(\cH),\quad z\in \bbC\backslash \bbR,\, \ell\in \bbN.
\end{equation}
Thus, Hypothesis \ref{hB.1} guarantees the existence of real-valued spectral shift functions $\xi\big(\,\cdot\,;T,T^{(0)}\big)$ and $\xi\big(\,\cdot\,;T_{\ell},T^{(0)}_{\ell}\big)$, $\ell\in \bbN$, which satisfy
\begin{equation}\lb{A.18}
\begin{split}
\tr_{\cH}\Big(\big(T-zI_{\cH}\big)^{-1} - \big(T^{(0)}-zI_{\cH}\big)^{-1} \Big) = -\int_{-\infty}^{\infty}\frac{\xi\big(\lambda;T,T^{(0)}\big)}{(\lambda - z)^2}\, d\lambda,&\\
\quad z\in \rho(T)\cap\rho\big(T^{(0)}\big),&
\end{split}
\end{equation}
and
\begin{equation}
\begin{split}
\tr_{\cH_{\ell}}\Big(\big(T_{\ell}-zI_{\cH}\big)^{-1} - \big(T^{(0)}_{\ell}-zI_{\cH}\big)^{-1} \Big) = -\int_{-\infty}^{\infty}\frac{\xi\big(\lambda;T_{\ell},T^{(0)}_{\ell}\big)}{(\lambda - z)^2}\, d\lambda,&\\
z\in \rho(T_{\ell})\cap\rho\big(T^{(0)}_{\ell}\big),\, \ell\in \bbN,&
\end{split}
\end{equation}
and are uniquely determined almost everywhere by the conditions
\begin{equation}
\int_{-\infty}^{\infty}\frac{|\xi\big(\lambda;T,T^{(0)}\big)|}{1+\lambda^2}\, \mathrm{d}\lambda < \infty\quad \text{and}\quad \xi\big(\lambda;T,T^{(0)}\big) = 0,\quad \lambda< \min \big[\sigma(T)\cup\sigma\big(T^{(0)}\big)\big],
\end{equation}
and for each $\ell\in \bbN$,
\begin{equation}
\int_{-\infty}^{\infty}\frac{|\xi\big(\lambda;T_{\ell},T_{\ell}^{(0)}\big)|}{1+\lambda^2}\, \mathrm{d}\lambda < \infty\quad \text{and}\quad \xi\big(\lambda;T_{\ell},T_{\ell}^{(0)}\big) = 0,\quad \lambda< \min \big[\sigma\big(T_{\ell}\big)\cup\sigma\big(T_{\ell}^{(0)}\big)\big].
\end{equation}
Moreover, Krein's trace formul\ae~hold:
\begin{align}
\tr_{\cH}\big(f(T)-f\big(T^{(0)}\big)\big) &= \int_{-\infty}^{\infty} f'(\lambda)\, \xi\big(\lambda;T,T^{(0)}\big)\, \mathrm{d}\lambda,\\
\tr_{\cH}\big(f(T_{\ell})-f\big(T^{(0)}_{\ell}\big)\big) &= \int_{-\infty}^{\infty} f'(\lambda)\, \xi\big(\lambda;T_{\ell},T^{(0)}_{\ell}\big)\, \mathrm{d}\lambda,\quad \ell\in \bbN,\lb{A.23}
\end{align}
where $\mathfrak{F}(\bbR)$ denotes the set of all function $f\colon \bbR\to \bbC$ with two locally bounded derivatives satisfying
\begin{equation}\lb{A.24}
\text{$\big(\lambda^2 f'(\lambda) \big)' = O\big(|\lambda|^{-1-\varepsilon} \big),\quad |\lambda|\to \infty$ for some $\varepsilon = \varepsilon(f) > 0,$}
\end{equation}
with
\begin{equation}\lb{A.25}
\lim_{|\lambda|\to \infty} f(\lambda) = C\quad \text{and}\quad \lim_{|\lambda|\to \infty} \lambda^2f'(\lambda) = D,
\end{equation}
for some constants $C=C(f), D=D(f)\in \bbC$.  In particular,
\begin{equation}
\text{$(\lambda - z)^{-n}\in \mathfrak{F}(\bbR)$, \quad$z\in \bbC\backslash \bbR$, $n\in \bbN$,}
\end{equation}
and in this case, the trace formula for $\big(T,T^{\free}\big)$ reads
\begin{equation}
\tr_{\cH}\big((T-zI_{\cH})^{-n} - \big(T^{\free}-zI_{\cH}\big)^{-n}\big) = -n \int_{\bbR} \frac{\xi\big(\lambda;T,T^{\free}\big)}{(\lambda - z)^{n+1}}\, d\lambda,\quad z\in \bbC\backslash \bbR,\, n\in \bbN,
\end{equation}
and an analogous formula holds for $\big(T_{\ell},T_{\ell}^{\free}\big)$, $\ell\in \bbN$.  We refer to \cite[Chapter 8]{Ya92} for additional information related to spectral shift functions.

Under the assumptions in Hypothesis \ref{hB.1}, the following convergence results hold for the sequence of spectral shift functions $\big\{\xi\big(\dott;T_{\ell},T_{\ell}^{(0)}\big)\big\}_{\ell=1}^{\infty}$.

%%%%%%%%%%%%
%\begin{theorem}[Theorem 3.10 in \cite{GN12b}]\lb{tB.10}
%Assume Hypothesis \ref{hB.1}.  If $g \in C_\infty(\bbR)$, then 
%\begin{equation}\lb{B.74}
%\lim_{\ell\to \infty} \int_{-\infty}^{\infty} \frac{\xi(\lambda;A_{\ell},T_{\ell}^{(0)})}{\lambda^2 + 1} \, g(\lambda) \, d\lambda
%= \int_{-\infty}^{\infty} \frac{\xi(\lambda;T,T^{(0)})}{\lambda^2 + 1} \, g(\lambda)\, d\lambda. 
%\end{equation} 
%\end{theorem}
%%%%%%%%%%%%

%%%%%%%%%%%
\begin{theorem}[Theorem 3.13 in \cite{GN12b}]\lb{tB.13}
Assume Hypothesis \ref{hB.1}. Then
\begin{equation}\lb{B.84}
\lim_{\ell\rightarrow \infty}\int_{-\infty}^{\infty} 
\frac{\xi\big(\lambda; T_{\ell},T_{\ell}^{(0)}\big)}{1+\lambda^2} \, f(\lambda) \, \mathrm{d}\lambda
= \int_{-\infty}^{\infty} \frac{\xi\big(\lambda;T,T^{(0)}\big)}{1+\lambda^2} \, f(\lambda)\, \mathrm{d}\lambda,  
\quad f\in C_b(\bbR).
\end{equation} 
\end{theorem}
%%%%%%%%%%%

The factor $(1+\lambda^2)^{-1}$ is essential in \eqref{B.84}, as the integrals need not exist without it.  One consequence of Theorem \ref{tB.13} is that the sequence $\big\{\xi\big(\dott;T_{\ell},T_{\ell}^{\free}\big)\big\}_{\ell=1}^{\infty}$ converges \textit{vaguley} to $\xi\big(\dott;T,T^{\free}\big)$.

%%%%%%%%%%
\begin{corollary}[Corollary 3.11 in \cite{GN12b}]\lb{cB.11}
Assume Hypothesis \ref{hB.1}. Then 
\begin{equation}\lb{3.82}
\lim_{\ell\to \infty} \int_{-\infty}^{\infty}\xi\big(\lambda;T_{\ell},T_{\ell}^{(0)}\big) \, g(\lambda) \, \mathrm{d}\lambda
= \int_{-\infty}^{\infty}\xi\big(\lambda;T,T^{(0)}\big) \, g(\lambda)\, \mathrm{d}\lambda\quad g \in C_c(\bbR).
\end{equation} 
\end{corollary}
%%%%%%%%%%

Finally, the continuity assumption in Theorem \ref{tB.13} may be slightly relaxed as follows.

%%%%%%%%%%%
\begin{corollary}[Corollary 3.14 in \cite{GN12b}]\lb{cB.14}
Assume Hypothesis \ref{hB.1}. Then the convergence in \eqref{B.84} holds for any bounded Borel measurable function that is continuous almost everywhere with respect to Lebesgue measure on $\bbR$.  In particular, 
\begin{equation}
\lim_{\ell\to \infty}\int_{S}\frac{\xi\big(\lambda;T_{\ell},T_{\ell}^{(0)}\big)}{1+\lambda^2}\,\mathrm{d}\lambda = \int_{S}\frac{\xi\big(\lambda;T,T^{(0)}\big)}{1+\lambda^2}\,\mathrm{d}\lambda
\end{equation}
holds for any set $S\subseteq\bbR$ that is boundaryless with respect to Lebesgue 
measure $($i.e., any set $S\subseteq\bbR$ for which the boundary of $S$ has Lebesgue measure 
equal to zero$)$.
\end{corollary}
%%%%%%%%%%%

%%%%%%%%%%%
\begin{corollary}[Corollary 3.15 in \cite{GN12b}]\lb{cB.15}
Assume Hypothesis \ref{hB.1}. If $g$ is a bounded Borel measurable function that is compactly supported and Lebesgue almost everywhere continuous on $\bbR$, then
\begin{equation}
\lim_{\ell\rightarrow \infty}\int_{-\infty}^{\infty} \xi\big(\lambda; T_{\ell},T_{\ell}^{(0)}\big)  \, g(\lambda) \, \mathrm{d}\lambda
= \int_{-\infty}^{\infty} \xi\big(\lambda;T,T^{(0)}\big) \, g(\lambda)\, \mathrm{d}\lambda.
\end{equation}
\end{corollary}
%%%%%%%%%%%

%%%%%%%%%%%%%%%%%%%%%%%%%%%%%%%%%%%%%
%%%%%%%%%%%%%%%%%%%%%%%%%%%%%%%%%%%%%

\medskip
 
%%%%%%%%%%%%%%%%%%%%%%%%%%%%%%%%%%%%%
\noindent {\bf Acknowledgments.}
The research of the authors was supported by the National Science Foundation under grant DMS--1852288.
%%%%%%%%%%%%%%%%%%%%%%%%%%%%%%%%%%%%%

%%%%%%%%%%%%%%%%%%%%%%%%%%%%%%%%
%%%%%%%%%%%%%%%%%%%%%%%%%%%%%%%%

\end{document}